\documentclass [a4paper]{article}
\usepackage[cp1251]{inputenc}
\usepackage{amsmath,amssymb,amsthm}
\usepackage{cite}
 \begin{document}

\title{Determinantal representations \\of the  weighted core-EP,  DMP, MPD, and  CMP inverses of matrices with quaternion and complex elements.}

\author {{\textbf{Ivan I. Kyrchei}} \thanks{Pidstrygach Institute for Applied Problems of Mechanics and Mathematics, NAS of Ukraine, Lviv,  Ukraine, E-mail address: kyrchei@online.ua}
}

\date{}
\maketitle
\begin{abstract}  In this paper, we extend notions of the weighted core-EP right and left inverses, the weighted DMP and MPD inverses, and the CMP inverse to matrices over the quaternion skew field  ${\mathbb{H}}$ that have some features in comparison to  these inverses over the complex field. We give the direct methods of their computing, namely,  their determinantal representations by using noncommutative column and row  determinants previously introduced by the author. As the special cases, by using the usual determinant, we give their determinantal representations for  matrices with complex entries as well. A numerical example to illustrate the main result is given.
\end{abstract} 



\section{Introduction}
\newtheorem{cor}{Corollary}[section]
\newtheorem{thm}{Theorem}[section]
\newtheorem{lem}{Lemma}[section]
\theoremstyle{definition}
 \newtheorem{defn}[thm]{Definition}
 \theoremstyle{remark}
 \newtheorem{rem}[thm]{Remark}
\newcommand{\rk}{\mathop{\rm rk}\nolimits}
\newcommand{\Ind}{\mathop{\rm Ind}\nolimits}

 \newtheorem{alg}[thm]{Algorythm}
 \numberwithin{equation}{section}

\newcommand{\corer}[1]{#1^{\tiny{\textcircled{\tiny\#}}}}
\newcommand{\corel}[1]{#1_{\tiny{\textcircled{\tiny\#}}}}
\newcommand{\coepr}[1]{#1^{\tiny{\textcircled{\tiny\dag}}}}
\newcommand{\coepl}[1]{#1_{\tiny{\textcircled{\tiny\dag}}}}

In the whole article, the notations ${\mathbb R}$ and $
{\mathbb{C}}$ are reserved for fields of the real and complex numbers, respectively.
$\mathbb{H}^{m\times n}$ stands for the set of all $m\times n$ matrices over the quaternion skew field
$$\mathbb{H}=\{h_0+h_1\mathbf{i}+h_2\mathbf{j}+h_3\mathbf{k}\mid \mathbf{i}^2=\mathbf{j}^2=\mathbf{k}^2=\mathbf{ijk}=-1,h_0,h_1,h_2,h_3\in\mathbb{R}\}.$$
$\mathbb{H}^{m\times n}_r$ determines its subset of matrices   with a rank $r$.
 For given
$h=h_{0}+h_{1}\mathbf{i}+h_{2}\mathbf{j}+h_{3}\mathbf{k}\in \mathbb{H}$, the conjugate of $h$ is  $\overline{h}=a_{0}-h_{1}\mathbf{i}-h_{2}\mathbf{j}-h_{3}\mathbf{k}$.
  For $ {\bf A}  \in
{\mathbb{H}}^{m\times n}$, the symbols  $ {\bf A}^{ *}$ and $\rk ({\bf A})$  specify  the conjugate transpose   and the  rank
of  $ {\bf A}$, respectively.
A matrix $ {\bf A}  \in
{\mathbb{H}}^{n\times n}$ is Hermitian if ${\rm {\bf
A}}^{ *}  = {\bf A}$.
The index of $ {\bf A}  \in
{\mathbb{H}}^{n\times n}$, denoted  $\Ind{\bf A}= k $, is the smallest positive number such that $\rk ({\bf A}^{k+1})=\rk ({\bf A}^{k})$.

Due to \cite{cl} the definition of the weighted Drazin inverse can be generalized over ${\mathbb{H}}$ as follows.
\begin{defn}\label{def:wdr}
For ${\bf A}\in{\rm {\mathbb{H}}}^{m\times n}$ and ${\bf W}\in{\rm {\mathbb{H}}}^{n\times m}$, the W-weighted Drazin inverse  of ${\bf A}$ with respect to ${\bf W}$, denoted by ${\bf A}^{d,W}$, is the unique solution to equations,
\begin{align*}
  ({{\bf A}{\bf W}})^{k+1}{\rm
{\bf X}}{\bf W}=({\rm {\bf
         A}}{\bf W})^{k},~
 {\bf X}{\bf W}{\bf A}{\bf W} {\bf X}= {\bf X},~
   {\bf A}{\bf W} {\bf X}= {\bf X}{\bf W} {\bf A},
\end{align*}
where $k= \Ind({\bf A}{\bf W})$.
\end{defn}
The properties of the complex W-weighted Drazin inverse can be found in \cite{sta,wei1,wei2,zhour}.
These properties can be generalized to ${\mathbb{H}}$.
Among them, if ${\bf A}\in {\mathbb{H}}^{m\times n}$ with respect to ${\bf W}\in {\mathbb{H}}^{n\times m}$ and $k= {\max}\{\Ind({\bf A}{\bf W}), \Ind({\bf W}{\bf A})\}$, then
   \begin{align}\label{eq:WDr_dr}
{\bf A}_{d,{\bf W}}=&{\bf A}\left(({\bf W}{\bf A})^{d} \right)^{2}=\left(({\bf A}{\bf W})^{d} \right)^{2}{\bf A}.
\end{align}

Let ${\bf A}\in {\mathbb{H}}^{n\times n}$ and ${\bf W}={\bf I}_n$ be the identity matrix of order $n$. Then ${\bf X}= {\bf A}^d$ is the Drazin
inverse of ${\bf A}$.
In particular, if $\Ind {\bf A}=1$,
then the matrix ${\bf X}$  is called
the group inverse and it is denoted by $ {\bf X}= {\bf
A}^{\#}$.

Using the Penrose equations \cite{pen}, the Moore-Penrose inverse of a quaternion matrix  can be defined as well (see, e.g.\cite{kyr}).
\begin{defn}\label{def:mp}
The Moore-Penrose inverse
 of $ {\bf A}
 \in  {\mathbb{H}}^{n\times m}$ is called the exclusive matrix ${\bf X}$, denoted by ${\bf A}^{\dagger}$,  satisfying the following four equations
\begin{align*}  {\bf A}{\bf X}
{\bf A} =  {\bf A},~ {\bf X}  {\bf
A}{\bf X}  = {\bf X},~ \left( {\bf A}{\bf X} \right)^{ *}  =
{\bf A}{\bf X},~ \left( {{\bf X} {\bf A}} \right)^{ *}  ={\bf X} {\bf A}.\end{align*}
\end{defn}

 ${\bf P}_A:= {\bf A}{\bf A}^{\dag}$ and ${\bf Q}_A:= {\bf A}^{\dag}{\bf A}$ are the orthogonal projectors onto the range of ${\bf A}$ and the range of ${\bf A}^*$, respectively.
For $ {\bf A} \in {\rm {\mathbb{C}}}^{n\times m}$, the symbols  $\mathcal{N} ({\bf A})$, and $\mathcal{R} ({\bf A})$ will denote  the kernel and the range space of ${\bf A}$, respectively.

The core inverse was introduced by Baksalary and Trenkler in \cite{baks}. Later, it was investigated by S. Malik in \cite{mal1} and S.Z. Xu et al. in \cite{xu},  among others.
\begin{defn}\cite{baks}\label{def:cor}
A matrix ${\bf X}\in  {\mathbb{C}}^{n\times n}$ is called the core inverse of ${\bf A} \in {\mathbb{C}}^{n\times n}$
if it satisfies the conditions
$$
{\bf A}{\bf X}={\bf P}_A,~and~ \mathcal{R}({\bf X})=\mathcal{R}({\bf A}).$$
When such matrix ${\bf X}$ exists, it is denoted $ \corer{\bf A}$.
\end{defn}
In 2014, the core inverse
was extended to the core-EP inverse defined by K. Manjunatha Prasad and  K.S. Mohana
\cite{pras}.
Determinantal formulas for the core EP generalized inverse in complex matrices has been derived in \cite{pras} based on the determinantal representation of an  reflexive inverse obtained in \cite{bap,rao}.

Other generalizations of the core inverse were recently introduced for $n\times n$ complex matrices,
namely BT inverses \cite{baks1},  DMP inverses \cite{mal1},  and CMP inverses \cite{meh}, etc.
The characterizations, computing methods, some applications of the core inverse and its generalizations were  investigated in complex matrices and rings (see, e.g.
\cite{chen,fer1,fer2,gao,gut,liu,ma,miel,mos,pras1,rak,wang}).

Only recently generalizations of the core inverse were extended to rectangular matrices by using the weighted Drazin inverse. Among them,   the W-weighted core-EP  inverse in complex matrices was introduced  in \cite{fer2}, its representations and properties were studied  in \cite{gao2}, and generalizations of the weighted core-EP inverse were expanded over  a ring with involution \cite{mos} and  Hilbert space \cite{mos1}, respectively. The concepts of the complex weighted DMP and CMP inverses were introduced in \cite{meng} and \cite{mos1}, respectively.

The main goals of this paper are extended the notions of the weighted core-EP inverses, and the  weighted  DMP  and CMP inverses over the quaternion skew-field ${\mathbb{H}}$, and get their determinantal representations  that are the direct methods of their obtaining by using determinants.

The determinantal representation of the usual inverse is the matrix with cofactors
in entries that suggests a direct method of finding the inverse of a matrix. The same
is desirable for the generalized inverses. But, there are various expressions of determinantal representations of generalized inverses even for matrices with complex or
real entries,   (see, e.g.  \cite{bap,rao,sta1,sta2,ky,kyr1,kyr_nov}).
In view of the noncommutativity of quaternions, the problem of the determinantal representation of quaternion generalized inverses is evidently dependent on complexities related with definition of  the  determinant  with noncommutative entries (it is also called  a noncommutative  determinant).

The majorities of the previous defined   noncommutative determinants  are derived by transforming the quaternion matrix to an equivalent complex or real matrix (see, e.g.\cite{as,coh}). However, by this way it is impossible for us to give determinantal representations of quaternionic generalized inverses.   Only now it can be done thanks to the theory of column-row determinants introduced by the author in \cite{kyr2,kyr3}.
Currently, by using of row-column determinants, determinantal representations of various generalized inverses  have been derived and applied to solutions of quaternion matrix equations by the author (see, e.g.\cite{kyr,kyr4,kyr5,kyr6,kyr7,kyr9,kyr10,kyr11,kyr12,kyr13,kyr14}), (among them the core inverse and its generalizations  in the quaternion \cite{ky_cor} and complex \cite{ky_cor1} cases), and by other researchers (see, e.g.\cite{song1,song2,song5}).

The paper is organized as follows.  In Section 2, we start with preliminary  introduction of the theory of row-column determinants and the  determinantal representations of the Moore-Penrose inverse, of the Drazin and weighted  Drazin  inverses, and of the core inverse and its generalizations over the quaternion skew field  previously obtained by using  row-column determinants.  In Section 3, we introduce the concepts of  the left and right weighted core-EP inverses over the quaternion skew field and  give their determinantal representations. In Section 4, the quaternion  weighted DMP and MPD  inverses   are established and their determinantal representations  are obtained.  Determinantal representations of  the quaternion CMP inverse are get in Section 5.
A numerical example to illustrate the main results is considered in   Section 6. Finally, in Section 7, the conclusions are drawn.

\section{Preliminaries.}

\subsection{Elements of the theory of  row-column determinants.}
Suppose $S_{n}$ is the symmetric group on the set $I_{n}=\{1,\ldots,n\}$.

\begin{defn}\cite{kyr2}
\emph{The $i$th row determinant} of ${\rm {\bf A}}=(a_{ij}) \in
{\mathbb{H}}^{n\times n}$ is defined  for any $i \in I_{n} $
by setting
 \begin{align*}{\rm{rdet}}_{ i} {\rm {\bf A}} =&
\sum\limits_{\sigma \in S_{n}} \left( { - 1} \right)^{n - r}({a_{i{\kern
1pt} i_{k_{1}}} } {a_{i_{k_{1}}   i_{k_{1} + 1}}} \ldots   {a_{i_{k_{1}
+ l_{1}}
 i}})  \ldots  ({a_{i_{k_{r}}  i_{k_{r} + 1}}}
\ldots  {a_{i_{k_{r} + l_{r}}  i_{k_{r}} }}),\\
\sigma =& \left(
{i\,i_{k_{1}}  i_{k_{1} + 1} \ldots i_{k_{1} + l_{1}} } \right)\left(
{i_{k_{2}}  i_{k_{2} + 1} \ldots i_{k_{2} + l_{2}} } \right)\ldots \left(
{i_{k_{r}}  i_{k_{r} + 1} \ldots i_{k_{r} + l_{r}} } \right),\end{align*}
where $\sigma$ is the left-ordered  permutation. It means that its first cycle from the left  starts with $i$, other cycles  start from the left  with the minimal of all the integers which are contained in it,
$$i_{k_{t}}  <
i_{k_{t} + s}~~ \text{for all}~~ t = 2,\ldots,r,~~~s =1,\ldots,l_{t}, $$
and  the order of disjoint cycles (except for the first one)  is strictly conditioned by increase from left to right of their first elements, $i_{k_{2}} < i_{k_{3}}  < \cdots < i_{k_{r}}$.
\end{defn}
Similarly, for a column determinant along an arbitrary column, we have the following definition.
\begin{defn}\cite{kyr2}
\emph{The $j$th column determinant}
 of ${\rm {\bf
A}}=(a_{ij}) \in
{\mathbb{H}}^{n\times n}$ is defined for
any $j \in I_{n} $ by setting
 \begin{align*}{\rm{cdet}} _{{j}}  {\bf A} =&
\sum\limits_{\tau \in S_{n}} ( - 1)^{n - r}(a_{j_{k_{r}}
j_{k_{r} + l_{r}} } \ldots a_{j_{k_{r} + 1} j_{k_{r}} })  \ldots  (a_{j
j_{k_{1} + l_{1}} }  \ldots  a_{ j_{k_{1} + 1} j_{k_{1}} }a_{j_{k_{1}}
j}),\\
\tau =&
\left( {j_{k_{r} + l_{r}}  \ldots j_{k_{r} + 1} j_{k_{r}} } \right)\ldots
\left( {j_{k_{2} + l_{2}}  \ldots j_{k_{2} + 1} j_{k_{2}} } \right){\kern
1pt} \left( {j_{k_{1} + l_{1}}  \ldots j_{k_{1} + 1} j_{k_{1} } j}
\right), \end{align*}
\noindent where $\tau$ is the right-ordered  permutation. It means that its first cycle from the right  starts with $j$, other cycles  start from the right  with the minimal of all the integers which are contained in it,
$$j_{k_{t}}  < j_{k_{t} + s} ~~ \text{for all}~~ t = 2,\ldots,r,~~~s =1,\ldots,l_{t}, $$
and the order of disjoint cycles (except for the first one) is strictly conditioned by increase from right to left of their first elements,
 $j_{k_{2}}  < j_{k_{3}}  < \cdots <
j_{k_{r}} $.
\end{defn}
The row and column determinants have the following linear properties.
\begin{lem}\cite{kyr2}\label{lem:left_rdet} If the $i$th row of
 $ {\bf A}\in {\mathbb{H}}^{n\times n}$ is a left linear combination
of  some row vectors, i.e. $ {\bf a}_{i.} = \alpha_{1} {\bf b}_{1 } + \cdots + \alpha_{k}  {\bf b}_{k }$, where $
\alpha_{l} \in { {\mathbb{H}}}$ and ${\bf b}_{l }\in {\mathbb{H}}^{1\times n}$ for all $ l = {1,\ldots, k}$ and $ i = {1,\ldots, n}$, then
\[
 {\rm{rdet}}_{i}\,  {\bf A}_{i  .} \left(
\alpha_{1}  {\bf b}_{1 } + \cdots + \alpha_{k}
{\bf b}_{k }  \right)=\sum_l  \alpha_{l}{\rm{rdet}}_{i}\, {\bf A}_{i  .} \left(
  {\bf b}_{l } \right).
\]
\end{lem}
\begin{lem}\cite{kyr2}\label{lem:right_cdet} If the $j$th column of
 $ A\in {\mathbb{H}}^{m\times n}$ is a right linear combination
of  other column vectors, i.e. $ {\bf a}_{.j} =  {\bf c}_{1}\alpha_{1} + \cdots +  {\bf c}_{k } \alpha_{k}$, where $
\alpha_{l} \in { {\mathbb{H}}}$ and ${\bf c}_{l }\in {\mathbb{H}}^{n\times1 }$ for all $ l = {1,\ldots, k}$ and $ j = {1,\ldots, n}$, then
\[
 {\rm{cdet}}_{j}\,  {\bf A}_{.j} \left(
  {\bf c}_{1 }\alpha_{1} + \cdots +
{\bf c}_{k }  \alpha_{k} \right)=\sum_l {\rm{cdet}}_{j}\,  {\bf A}_{.j} \left(
  {\bf c}_{l} \right) \alpha_{l}.
\]
\end{lem}
So,  an arbitrary $n\times n$ quaternion matrix inducts  $n$ row determinants and $n$ column determinants that are different in general. Only for a
    Hermitian matrix $ {\bf A}$, we have \cite{kyr2},
 $${\rm{rdet}} _{1}  {\bf A} = \cdots = {\rm{rdet}} _{n} {\bf
A} = {\rm{cdet}} _{1}  {\bf A} = \cdots = {\rm{cdet}} _{n}  {\bf
A} \in  {\mathbb{R}},$$ that enables to define \emph{the determinant
of a Hermitian matrix}  by setting
$\det {\bf A}: = {\rm{rdet}}_{{i}}\,
{\bf A} = {\rm{cdet}} _{{i}}\, {\bf A} $
 for all $i =1,\ldots,n$.
Its properties  have been completely studied  in \cite{kyr3}. In particular, from them it follows the definition of the \emph{determinantal rank} of a quaternion matrix $ {\bf A} $ as the largest possible size of  nonzero principal minors of its corresponding Hermitian matrices, i.e. $\rk{\bf A}=\rk({\bf A}^*{\bf A})=\rk({\bf A}{\bf A}^*)$.

\subsection{Determinantal representations of generalized inverses.}

Let $\alpha : = \left\{
{\alpha _{1},\ldots,\alpha _{k}} \right\} \subseteq {\left\{
{1,\ldots ,m} \right\}}$ and $\beta : = \left\{ {\beta _{1}
,\ldots ,\beta _{k}} \right\} \subseteq {\left\{ {1,\ldots ,n}
\right\}}$ be subsets with $1 \le k \le \min {\left\{
{m,n} \right\}}$.
By ${\bf A}_{\beta} ^{\alpha} $ denote a submatrix of $ {\bf A}
 \in  {\mathbb{H}}^{m\times n}$  with rows and columns  indexed by $\alpha$ and
 $\beta$, respectively. Then, $ {\bf A}_{\alpha} ^{\alpha} $ is a principal submatrix of $ {\bf A}$ with rows and columns
indexed by $\alpha$. Moreover, for Hermitian $ {\bf A}$,
$ |{\bf A}|_{\alpha} ^{\alpha} $ is the
corresponding principal minor of $\det  {\bf A}$.
Suppose that
  $$\textsl{L}_{ k,
n}: = {\left\{ {\alpha :\alpha = \left( {\alpha_{1},\ldots,\alpha_{k}} \right),\,\, 1 \le \alpha_{1} < \cdots< \alpha_{k} \le n} \right\}}$$ stands for the collection of strictly
increasing sequences of  $1 \leq k\leq n$ integers chosen from $\left\{
{1,\ldots,n} \right\}$.  For fixed $i \in \alpha $ and $j \in
\beta $, put $I_{r, m} {\left\{ {i} \right\}}: = {\left\{
{\alpha :\alpha \in L_{r, m}, i \in \alpha}  \right\}}$,
$J_{r, n} {\left\{ {j} \right\}}: = {\left\{ {\beta :\beta
\in L_{r, n}, j \in \beta}  \right\}}$.

 Denote by $ {\bf a}_{.\,j} $ and $ {\bf
a}_{i.} $  the $j$-th column and the $i$-th row of $ {\bf A}$.
Similarly, $ {\bf a}_{.j}^{*} $ and $ {\bf
a}_{i.}^{*} $ stand for the $j$-th column  and the $i$-th row of  $
{\bf A}^{*} $.
By  ${\bf A}_{i.} \left(  {\bf b} \right)$ and $ {\bf
A}_{.j} \left(  {\bf c} \right)$ we denote the matrices obtained from
$ {\bf A}$ by replacing its $i$-th row with the
row ${\bf b}$ and its $j$-th column with the column ${\bf
c}$, respectively.

\begin{thm} \cite{kyr4}\label{th:det_rep_mp}
If $ {\bf A} \in  {\mathbb{H}}_{r}^{m\times n} $, then
the Moore-Penrose inverse  ${\rm {\bf A}}^{ \dag} = \left( {a_{ij}^{
\dag} } \right) \in  {\mathbb{H}}_{}^{n\times m} $ possess the
determinantal representations
  \begin{align}
\label{eq:cdet_repr_A*A}
 a_{ij}^{ \dag}  =& {\frac{{{\sum\limits_{\beta
\in J_{r,n} {\left\{ {i} \right\}}} {{\rm{cdet}} _{i} \left(
{\left( { {\bf A}^{ *} {\bf A}} \right)_{. i}
\left( { {\bf a}_{\,.j}^{ *} }  \right)} \right)
 _{\beta} ^{\beta} } } }}{{{\sum\limits_{\beta \in
J_{r,n}} {{\left|  { {\bf A}^{ *}  {\bf A}}
\right| _{\beta} ^{\beta}  }}} }}} =\\=&\label{eq:rdet_repr_AA*}
{\frac{{{\sum\limits_{\alpha \in I_{r,m} {\left\{ {j} \right\}}}
{{\rm{rdet}} _{j} \left( {( {\bf A} {\bf A}^{ *}
)_{j.} ( {\bf a}_{\,i.}^{ *} )} \right)_{\alpha}
^{\alpha} } }}}{{{\sum\limits_{\alpha \in I_{r,m}}  {{
{\left| { {\bf A} {\bf A}^{ *} } \right|
_{\alpha} ^{\alpha} } }}} }}}.
\end{align}
\end{thm}
\begin{rem}\label{rem:unit_repr} For an arbitrary full-rank matrix $ {\bf A} \in  {\mathbb{H}}_{r}^{m\times n} $, a row-vector ${\bf b}\in  {\mathbb{H}}^{1\times m} $, and a column-vector ${\bf c}\in  {\mathbb{H}}^{n\times 1}$,  we put, respectively,
\begin{itemize}
  \item when $r=m$
 \begin{align*} {\rm{rdet}} _{i} \left( {( {\bf A} {\bf A}^{ *}
)_{i.} \left( {\bf b}  \right)} \right)=&
{\sum\limits_{\alpha \in I_{m,m} {\left\{ {i} \right\}}}
{{\rm{rdet}} _{i} \left( {( {\bf A} {\bf A}^{ *}
)_{i.} \left( {\bf b}  \right)} \right)_{\alpha}
^{\alpha} } },\\
\det\left( { {\bf A} {\bf A}^{ *} } \right)=&{\sum\limits_{\alpha \in I_{m,m}}  {{
{\left| { {\bf A} {\bf A}^{ *} } \right|
_{\alpha} ^{\alpha} } }}},~~~~i=1,\ldots,m;
 \end{align*}

 \item
   when $r=n$
  \begin{align*}{\rm{cdet}} _{j} \left(
{\left( { {\bf A}^{ *} {\bf A}} \right)_{. j}
\left( {\bf c}  \right)} \right)=&
\sum\limits_{\beta
\in J_{n,n} {\left\{ {j} \right\}}} {{\rm{cdet}} _{j} \left(
{\left( { {\bf A}^{ *}  {\bf A}} \right)_{. j}
\left( {\bf c}  \right)} \right)
 _{\beta} ^{\beta} },
 \\\det\left(  { {\bf A}^{ *}  {\bf A}}
\right)=&{\sum\limits_{\beta \in
J_{n,n}} {{\left|  { {\bf A}^{ *} {\bf A}}
\right|_{\beta} ^{\beta}  }}},~~~~j=1,\ldots,n. \end{align*}

\end{itemize}
\end{rem}

\begin{cor}\label{cor:det_repr_proj_Q}
If $ {\bf A} \in  {\mathbb{H}}_{r}^{m\times n} $,  then the
determinantal representations of the
projection matrices $ {\bf A}^{ \dag}  {\bf A} = : {\bf
Q}_{A} = \left( {q^A_{ij}} \right)_{n\times n} $ and $ {\bf A}  {\bf A}^{ \dag} = : {\bf
P}_{A} = \left( {p^A_{ij}} \right)_{m\times m} $can be expressed as follows
  \begin{align}\label{eq:det_repr_proj_Q}
q^A_{ij} =  {\frac{{{\sum\limits_{\beta \in J_{r,n} {\left\{ {i}
\right\}}} {{\rm{cdet}} _{i} \left( {\left( {{\bf A}^{ *}
{\bf A}} \right)_{.i} \left({\bf \dot{a}}_{.j} \right)}
\right)  _{\beta} ^{\beta} } }}}{{{\sum\limits_{\beta
\in J_{r,n}}  {{ \left| {{\bf A}^{ *}  {\bf
A}} \right|_{\beta}^{\beta} }}}} }} = {\frac{{{\sum\limits_{\alpha \in I_{r,n} {\left\{ {j}
\right\}}} {{\rm{rdet}} _{j} \left( {\left( {{\bf A}^{ *}
{\bf A}} \right)_{.j} \left({\bf \dot{a}}_{.i} \right)}
\right)  _{\alpha} ^{\alpha} } }}}{{{\sum\limits_{\alpha
\in I_{r,n}}  {{ \left| {{\bf A}^{ *}  {\bf
A}} \right|_{\alpha}^{\alpha} }}}} }},\\
\label{eq:det_repr_proj_P}
p^A_{ij} = {\frac{{{\sum\limits_{\alpha \in I_{r,m} {\left\{ {j}
\right\}}} {{{\rm{rdet}} _{j} {\left( {({\bf A} {\bf A}^{ *}
)_{j .} ({\bf \ddot{a}}  _{i  .} )}
\right)  _{\alpha} ^{\alpha} } }}}
}}{{{\sum\limits_{\alpha \in I_{r,m}} {{ {\left| {
{\bf A} {\bf A}^{ *} } \right| _{\alpha
}^{\alpha} }  }}} }}}= {\frac{{{\sum\limits_{\beta \in J_{r,m} {\left\{ {i}
\right\}}} {{{\rm{cdet}} _{i} {\left( {({\bf A} {\bf A}^{ *}
)_{.\,i } ({\bf \ddot{a}}  _{.j  } )}
\right)  _{\beta} ^{\beta} } }}}
}}{{{\sum\limits_{\beta \in J_{r,m}} {{ {\left| {
{\bf A} {\bf A}^{ *} } \right| _{\beta
}^{\beta} }  }}} }}},
 \end{align}
\noindent where $ {\bf \dot{a}}_{.j} $ and ${\bf \dot{a}}_{.i}$,    ${\bf \ddot{a}} _{i.} $ and ${\bf \ddot{a}}  _{.j  }$ are the $i$-th rows and the $j$-th columns of
${ {\bf A}^{ *}  {\bf A}} \in
{\mathbb{H}}^{n\times n}$ and
${\bf A}{\bf A}^{*}\in  {\mathbb{H}}^{m\times m}$, respectively.
\end{cor}
The following corollary gives determinantal representations of the Moore-Penrose inverse and of  both projectors in complex matrices.
\begin{cor} \cite{kyr}\label{cor:det_repr_MP_c}
Let $ {\bf A} \in  {\mathbb{C}}_{r}^{m\times n} $. Then the
following determinantal representations are obtained
\begin{enumerate}
  \item[(i)] for the Moore-Penrose inverse  ${\rm {\bf A}}^{ \dag} = \left( {a_{ij}^{
\dag} } \right)_{n\times m} $,
  \begin{equation*}
 a_{ij}^{ \dag}  = {\frac{{{\sum\limits_{\beta
\in J_{r,n} {\left\{ {i} \right\}}} { \left|
{\left( { {\bf A}^{ *} {\bf A}} \right)_{. i}
\left( { {\bf a}_{\,.j}^{ *} }  \right)} \right|
 _{\beta} ^{\beta} } } }}{{{\sum\limits_{\beta \in
J_{r,n}} {{\left|  { {\bf A}^{ *}  {\bf A}}
\right| _{\beta} ^{\beta}  }}} }}}=
{\frac{{{\sum\limits_{\alpha \in I_{r,m} {\left\{ {j} \right\}}}
{ \left| {( {\bf A} {\bf A}^{ *}
)_{j.} ( {\bf a}_{\,i.}^{ *} )} \right|_{\alpha}
^{\alpha} } }}}{{{\sum\limits_{\alpha \in I_{r,m}}  {{
{\left| { {\bf A} {\bf A}^{ *} } \right|
_{\alpha} ^{\alpha} } }}} }}};
\end{equation*}
  \item[(ii)] for the
projector $ {\bf
Q}_{A} = \left( {q_{ij}} \right)_{n\times n} $,
  \begin{equation*}
q_{ij} = {\frac{{{\sum\limits_{\beta \in J_{r,n} {\left\{ {i}
\right\}}} { \left| {\left( {{\bf A}^{ *}
{\bf A}} \right)_{.i} \left({\bf \dot{a}}_{.j} \right)}
\right|  _{\beta} ^{\beta} } }}}{{{\sum\limits_{\beta
\in J_{r,n}}  {{ \left| {{\bf A}^{ *}  {\bf
A}} \right|_{\beta}^{\beta} }}}} }},
 \end{equation*}
where $ {\bf \dot{a}}_{.j} $ is the $j$th column of
${ {\bf A}^{ *}  {\bf A}}$;
  \item[(iii)] for the
projector $  {\bf
P}_{A} = \left( {p_{ij}} \right)_{m\times m} $,
 \begin{equation*}
p_{ij} = {\frac{{{\sum\limits_{\alpha \in I_{r,m} {\left\{ {j}
\right\}}} {{ {\left| {({\bf A} {\bf A}^{ *}
)_{j .} ({\bf \ddot{a}}  _{i  .} )}
\right|  _{\alpha} ^{\alpha} } }}}
}}{{{\sum\limits_{\alpha \in I_{r,m}} {{ {\left| {
{\bf A} {\bf A}^{ *} } \right| _{\alpha
}^{\alpha} }  }}} }}},
 \end{equation*}
where ${\bf \ddot{a}} _{i.} $ is the $i$th row of $
{\bf A}{\bf A}^{*}$.
\end{enumerate}
\end{cor}

There are two case for determinantal representations of the W-weighted Drazin inverse over the quaternion skew field.
\begin{lem}\cite{kyr6}\label{theor:det_rep_wdraz1}
 Let ${\bf A} \in  {\mathbb{H}}^{m\times n}$, ${\bf W}\in {\mathbb{H}}^{n\times m}$,  $k= {\max}\{\Ind\left({\bf A}{\bf W}\right), \Ind\left({\bf W}{\bf A}\right)\}$. Denote ${\bf A}{\bf W}={\bf V}=\left(v_{ij}
\right)  \in {\mathbb{H}}^{m\times m}$ and ${\bf W}{\bf A}={\bf
U}=\left(u_{ij} \right)\in {\mathbb{H}}^{n\times n}$. Then for  $ {\bf
A}_{d,W} = \left( {a_{ij}^{d,W} } \right) \in
{\mathbb{H}}^{m\times n} $, we have
\begin{itemize}
  \item[(i)] if $\rk{\bf U}^{k+1} =
\rk{\bf U}^{k} = r$,

\begin{equation}\label{eq:det_rep_u}
a_{ij}^{d,W}=
{\frac{\sum\limits_{s = 1}^{n}\left({{\sum\limits_{\alpha \in I_{r,\,n} {\left\{ {s}
\right\}}} {{\rm{rdet}} _{s} \left( {\left( { {\bf U}^{ 2k+1} \left({\bf U}^{ 2k+1} \right)^{*}
} \right)_{s.} (\widetilde{ {\bf \phi}}_{i\,.})} \right)  _{\alpha} ^{\alpha} } }
}\right){u}_{sj}^{(k)}}{{\left(
{\sum\limits_{\alpha \in I_{r,\,n}} {{\left|  {\bf U}^{ 2k+1} \left({\bf U}^{ 2k+1} \right)^{*}
 \right|_{\alpha} ^{\alpha}}}}\right)^2 }}},
\end{equation}
where  $\widetilde{ {\bf \phi}}_{i\,.}$ is the $i$-th row of $
\widetilde{ {\bf \Phi}}:={\bf A}{\bf \Phi}{\bf U}^{2k}({\bf U}^{ 2k+1})^{*}\in {\mathbb{H}}^{m\times n}$, and $
{\bf \Phi}=(\phi_{lq})\in  {\mathbb{H}}^{n\times n}$ such that

$$
\phi_{lq}={\sum\limits_{\alpha \in I_{r,\,n} {\left\{ {q}
\right\}}} {{\rm{rdet}} _{q} \left( {\left( { {\bf U}^{ 2k+1} \left({\bf U}^{ 2k+1} \right)^{*}
} \right)_{q.} (\check{ {\bf u}}_{l\,.})} \right)  _{\alpha} ^{\alpha} } }.
$$
Here
$\check{ {\bf u}}_{l.}$ is the $l$-th row of $
{\bf U}^{k}({\bf U}^{ 2k+1})^{*} =:\check{ {\bf U}}\in {\mathbb{H}}^{n\times n}$;
  \item [(ii)] if $\rk{\bf V}^{k+1} =
\rk{\bf V}^{k} = r$,

\begin{equation}\label{eq:det_rep_v}
a_{ij}^{d,W}=
 {\frac{{ \sum\limits_{t = 1}^{m} {v}_{it}^{(k)}   {\sum\limits_{\beta \in J_{r,\,m} {\left\{ {t}
\right\}}} {{\rm{cdet}} _{t} \left( {\left(\left({\bf V}^{ 2k+1} \right)^{*}{\bf V}^{ 2k+1} \right)_{. \,t} \left( \widetilde{ {\bf \psi}}_{.\,j}
\right)} \right)  _{\beta} ^{\beta} } }
}}{{\left({\sum\limits_{\beta \in J_{r,\,m}} {{\left| \left({\bf V}^{ 2k+1} \right)^{*}{\bf V}^{ 2k+1}
  \right|_{\beta} ^{\beta}}}} \right)^2}}}
\end{equation}
where  $\widetilde{ {\bf \psi}}_{.\,j}$ is the $j$-th column of $
\widetilde{ {\bf \Psi}}:=({\bf V}^{ 2k+1})^{*}{\bf V}^{2k}{\bf \Psi}{\bf A}\in {\mathbb{H}}^{m\times n}$, and $
{\bf \Psi}=(\psi_{st})\in  {\mathbb{H}}^{m\times m}$ such that

$$
\psi_{st}={\sum\limits_{\beta \in J_{r,\,m} {\left\{ {s}
\right\}}} {{\rm{cdet}} _{s} \left( {\left(\left({\bf V}^{ 2k+1} \right)^{*}{\bf V}^{ 2k+1} \right)_{. s} \left( \hat{ {\bf v}}_{.t}
\right)} \right)  _{\beta} ^{\beta} } }
$$
Here
$\hat{ {\bf v}}_{.t}$ is the $t$-th column of  $
({\bf V}^{ 2k+1})^{*}{\bf V}^{k} =:\hat{ {\bf V}}\in {\mathbb{H}}^{m\times m}$.
\end{itemize}

\end{lem}

\begin{lem}\cite{kyr9}\label{theor:det_rep_wdrazh}
  Let ${\bf A} \in  {\mathbb{H}}^{m\times n}$, ${\bf W}\in {\mathbb{H}}^{n\times m}$,  $k= {\max}\{\Ind\left({\bf A}{\bf W}\right), \Ind\left({\bf W}{\bf A}\right)\}$.
Then for  $ {\bf
A}_{d,W} = \left( {a_{ij}^{d,W} } \right) \in
{\mathbb{H}}^{m\times n} $, we have
\begin{itemize}
  \item [(i)] if the matrix ${\bf A}{\bf W}\in {\mathbb{H}}^{m\times m}$ is Hermitian  and $\rk({\bf A}{\bf W})^{k+1} =
\rk({\bf A}{\bf W})^{k} = r$, then

\begin{equation}
\label{eq:dr_rep_wcdet} a_{ij}^{d,W}  = {\frac{\sum\limits_{\beta
\in J_{r,\,m} {\left\{ {i} \right\}}} {{\rm{cdet}} _{i} \left(
{\left( {{\rm {\bf A}}{\bf W}} \right)^{k+2}_{. \,i} \left( {{\rm {\bf
\bar{v}}}_{.j} }  \right)} \right) _{\beta}
^{\beta} }  }{\sum\limits_{\beta \in J_{r,m}} {{\left|
{\left( { {\bf A}{\bf W}} \right)^{k+2}}  \right|_{\beta} ^{\beta}}} }},
\end{equation}
where ${\rm {\bf \bar{v}}}_{.j} $ is the $j$-th column of  ${\bf \bar{V}}=({\bf A}{\bf W})^{k}{\bf A} $ for all $j=1,\ldots,n$.
  \item[(ii)]
if  ${\bf W}{\bf A}\in {\mathbb{H}}^{n\times n}$ is Hermitian  and $\rk({\bf W}{\bf A})^{k+1} =
\rk({\bf W}{\bf A})^{k} = r$, then
\begin{equation}
\label{eq:dr_rep_wrdet} a_{ij}^{d,W}  = {\frac{\sum\limits_{\alpha
\in I_{r,n} {\left\{ {j} \right\}}} {{\rm{rdet}} _{j} \left(
{({\bf W}{\rm {\bf A}} )^{ k+2}_{j.} ( {\bf \bar{u}}_{i.} )}
\right)_{\alpha} ^{\alpha} } }{\sum\limits_{\alpha \in
I_{r,n}}  {{\left| {\left({\bf W} {{\rm {\bf A}} } \right)^{k+2}   } \right|_{\alpha} ^{\alpha}}}} }.
\end{equation}

where  $ {\bf
\bar{u}}_{i.} $ is  the $i$-th row of  ${\bar
{ \bf{U}}}={\bf A}({\bf W}{\bf A})^{k} $ for all $i={1,\ldots,m}$.

\end{itemize}
\end{lem}

The following corollary gives determinantal representations of the W-weighted Drazin inverse  in complex matrices.

\begin{cor}\cite{kyr_nov}\label{theor:det_rep_wdrazc}
  Let ${\bf A} \in  {\mathbb{H}}^{m\times n}$, ${\bf W}\in {\mathbb{H}}^{n\times m}$,  $k= {\max}\{\Ind\left({\bf A}{\bf W}\right), \Ind\left({\bf W}{\bf A}\right)\}$.
Then   $ {\bf
A}_{d,W} = \left( {a_{ij}^{d,W} } \right) \in
{\mathbb{H}}^{m\times n} $ can be expressed as
\begin{align*}
 a_{ij}^{d,W}  =& {\frac{\sum\limits_{\beta
\in J_{r,\,m} {\left\{ {i} \right\}}} { \left(
{\left| {{\rm {\bf A}}{\bf W}} \right)^{k+2}_{. \,i} \left( {{\rm {\bf
\bar{v}}}_{.j} }  \right)} \right| _{\beta}
^{\beta} }  }{\sum\limits_{\beta \in J_{r,m}} {{\left|
{\left( { {\bf A}{\bf W}} \right)^{k+2}}  \right|_{\beta} ^{\beta}}} }}= {\frac{\sum\limits_{\alpha
\in I_{r,n} {\left\{ {j} \right\}}} { \left|
{({\bf W}{\rm {\bf A}} )^{ k+2}_{j.} ( {\bf \bar{u}}_{i.} )}
\right|_{\alpha} ^{\alpha} } }{\sum\limits_{\alpha \in
I_{r,n}}  {{\left| {\left({\bf W} {{\rm {\bf A}} } \right)^{k+2}   } \right|_{\alpha} ^{\alpha}}}} }.
\end{align*}
where  $ {\bf
\bar{u}}_{i.} $ is  the $i$-th row of  ${\bar
{ \bf{U}}}={\bf A}({\bf W}{\bf A})^{k} $  and ${\rm {\bf \bar{v}}}_{.j} $ is the $j$-th column of  ${\bf \bar{V}}=({\bf A}{\bf W})^{k}{\bf A} $.
\end{cor}
Quaternion column-vectors  form a right  vector $\mathbb{H}$-space
with quaternion-scalar right-multiplying,  and quaternion row-vectors  form a left  vector $\mathbb{H}$-space with quaternion-scalar left-multiplying. We denote them  by $\mathcal{H}_{r}$ and $\mathcal{H}_{l}$, respectively. Moreover, $\mathcal{H}_{r}$ and $\mathcal{H}_{l}$ possess corresponding $\mathbb{H}$-valued inner products by putting
 \begin{align*}\langle \mathbf{x},\mathbf{y}\rangle_{r}=&\overline{y}_1x_{1}+\cdots+\overline{y}_{n}x_{n}\,\,\mbox{for}\,\, \mathbf{x}=\left(x_{i}\right)_{i=1}^{n}, \mathbf{y}=\left(y_{i}\right)_{i=1}^{n}\in \mathcal{H}_{r},\\
\langle \mathbf{x},\mathbf{y}\rangle_{l}=&x_{1}\overline{y}_1+\cdots+x_{n}\overline{y}_{n}\,\, \mbox{for}\,\, \mathbf{x}, \mathbf{y}\in \mathcal{H}_{l},\end{align*} that satisfy the  inner product relations, namely, conjugate symmetry, linearity, and positive-definiteness but with specialties
 \begin{align*}\langle \mathbf{x}\alpha+\mathbf{y}\beta,\mathbf{z} \rangle=\langle \mathbf{x},\mathbf{z}\rangle\alpha+\langle \mathbf{y},\mathbf{z}\rangle\beta \,\,\mbox{ when}\,\, \mathbf{x}, \mathbf{y}, \mathbf{z} \in \mathcal{H}_{r},\\
          \langle \alpha\mathbf{x}+\beta\mathbf{y},\mathbf{z} \rangle=\alpha\langle \mathbf{x},\mathbf{z}\rangle+\beta\langle \mathbf{y},\mathbf{z}\rangle \,\, \mbox{ when}\,\, \mathbf{x}, \mathbf{y}, \mathbf{z} \in \mathcal{H}_{l},\end{align*}
 for any $\alpha, \beta \in {\mathbb{H}}$.

So, an arbitrary quaternion matrix induct right and left vector  $\mathbb{H}$-spaces that introduced  by the following definition.
\begin{defn}
For an arbitrary matrix over the quaternion skew field, ${\bf A}\in  {\mathbb{H}}^{m\times n}$, we denote by
\begin{itemize}
  \item  $\mathcal{C}_{r}({\rm {\bf A}})=\{ {\bf y}\in {\mathbb{H}}^{m\times 1} : {\bf y} = {\bf A}{\bf x},\,  {\bf x} \in {\mathbb{H}}^{n\times 1}\},$  the right column  space of ${\bf A}$,
  \item  $\mathcal{N}_{r}({\rm {\bf A}})=\{ {\bf x}\in {\mathbb{H}}^{n\times 1} : \,\, {\bf A}{\bf x}=0\}$,  the right null space of  ${\bf A}$,
  \item $\mathcal{R}_{l}({\rm {\bf A}})=\{ {\bf y}\in {\mathbb{H}}^{1\times n} : \,\,{\bf y} = {\bf x}{\bf A},\,\,  {\bf x} \in {\mathbb{H}}^{1\times m}\}$, the  left row space of ${\bf A}$,
  \item  $\mathcal{N}_{l}({\rm {\bf A}})=\{ {\bf x}\in {\mathbb{H}}^{1\times m} : \,\, {\bf x}{\bf A}=0\}$,  the left null space of  ${\bf A}$.
\end{itemize}
\end{defn}

\subsection{Determinantal representations of the core inverses and the core-EP inverses}

Because of  quaternion noncommutativity, Definition \ref{def:cor}  can be expand to matrices over ${\mathbb{H}}$ as follows.
\begin{defn}
A matrix ${\bf X} \in  {\mathbb{H}}^{n\times n}$ is said  to be \emph{the right core inverse} of ${\bf A}  \in  {\mathbb{H}}^{n\times n}$
if it satisfies the conditions
$$
{\bf A}{\bf X}={\bf P}_A,~and~ \mathcal{C}_{r}({\bf X})=\mathcal{C}_{r}({\bf A}).$$
When such matrix ${\bf X}$ exists, it is denoted $ \corer{\bf A}$.
\end{defn}
\begin{defn}\label{def:lcor}
A matrix ${\bf X} \in  {\mathbb{H}}^{n\times n}$  is said  to be \emph{the left core inverse} of $ {\bf A} \in  {\mathbb{H}}^{n\times n}$
if it satisfies the conditions
$$
{\bf X}{\bf A}={\bf Q}_A,~and~ \mathcal{R}_l({\bf X})=\mathcal{R}_l({\bf A}).$$
When such matrix ${\bf X}$ exists, it is denoted $ \corel{\bf A}$.
\end{defn}

Similar as in \cite{pras}, we introduce two  core-EP inverses over quaternion skew field.
\begin{defn}
A matrix ${\bf X}\in  {\mathbb{H}}^{n\times n}$ is said to be \emph{the right core-EP inverse} of $ {\bf A}\in  {\mathbb{H}}^{n\times n}$
if it satisfies the conditions
$$
{\bf X}{\bf A}{\bf X}={\bf A},~and~ \mathcal{C}_{r}({\bf X})=\mathcal{C}_{r}({\bf X}^*)=\mathcal{C}_{r}({\bf A}^d).$$
It is denoted $ \coepr{\bf A}$.
\end{defn}

The lemma below give characterization of the right core-EP inverse.
Due to \cite{pras}, the right weighted core-EP inverse is characterized in terms
of three equations.
\begin{lem}Let ${\bf A},~{\bf X}\in  {\mathbb{H}}^{n\times n}$ be such that $\Ind({\bf A}) = k$. Then ${\bf X}$ is the right core-EP inverse
of ${\bf A}$ if and only if ${\bf X}$ satisfies the conditions:
\begin{align*}{\bf X}{\bf A}^{k+1}= {\bf
         A}^{k},~ {\bf
A}{\bf X}^2  = {\bf X},~ \left( {\bf A}{\bf X} \right)^{ *}  =
{\bf A}{\bf X}~ \text{and}~~ \mathcal{C}_{r}({\bf X})\subseteq\mathcal{C}_{r}({\bf A}^k).\end{align*}
\end{lem}
Taking to account  (\cite{gao}, Theorem 2.3), the following expression can be extend to quaternion matrices.
\begin{lem} Let $ {\bf A}\in{\mathbb{H}}^{n\times n}$ and let $l$ be a non-negative integer such that
$l\geq k = \Ind({\bf A})$. Then $ {\coepr{\bf A}}={\bf A}^d{\bf A}^l\left({\bf A}^l\right)^{\dag}$.
\end{lem}

\begin{defn}\label{def:lepcor}
A matrix ${\bf X}\in  {\mathbb{H}}^{n\times n}$ is said to be  \emph{the left core-EP inverse} of ${\bf A}\in  {\mathbb{H}}^{n\times n}$
if it satisfies the conditions
$$
{\bf X}{\bf A}{\bf X}={\bf A},~and~ \mathcal{R}_l({\bf X})=\mathcal{R}_l({\bf X}^*)=\mathcal{R}_l({\bf A}^d).$$
It is denoted $ \coepl{\bf A}$.
\end{defn}
\begin{rem}Since $\mathcal{C}_{r}(({\bf A}^*)^d)=\mathcal{R}_l({\bf A}^d)$, then the left core inverse $ \coepl{\bf A}$ of  $ {\bf A}\in{\mathbb{C}}^{n\times n}$ is similar to  the  $*$core inverse introduced in \cite{pras}, and  the dual core-EP inverse introduced in \cite{zh_arx}.
\end{rem}
Similarly, we have   the following characterization of the left core-EP inverse.
\begin{thm}\label{th1:lcep}Let ${\bf X},~ {\bf A}\in{\mathbb{H}}^{n\times n}$ and let $l$ be a non-negative integer such that
$l\geq k = \Ind({\bf A})$. The following statements are
equivalent:
\begin{itemize}
  \item [(i)]  ${\bf X}$ is the left core-EP inverse of $ {\bf A}$.
  \item [(ii)] ${\bf A}^{k+1}{\bf X}= {\bf
         A}^{k},~{\bf X}^2  {\bf
A}  = {\bf X},~ \left( {\bf X}{\bf A} \right)^{ *}  =
{\bf X}{\bf A}~ \text{and}~~ \mathcal{R}_{l}({\bf X})\subseteq\mathcal{R}_{l}({\bf A}^k)$.
\item [(iii)] $ {\bf X}={\coepl{\bf A}}=\left({\bf A}^l\right)^{\dag}{\bf A}^l{\bf A}^d$.

\end{itemize}
\end{thm}

Thanks to \cite{gao}, there exists the simple relation between the left and right core-EP inverses, $({\coepr{\bf A}})^*={({\bf A}^*)_{\tiny{\textcircled{\tiny\dag}}}}$.  So, it is enough to investigate the left  core-EP inverse,
and right core-EP inverse case can be investigated analogously. But in \cite{ky_cor},  we gave separately determinantal representations of both core-EP inverses.

\begin{thm}\cite{ky_cor}\label{th:lcep}Suppose ${\bf A}\in  {\mathbb{H}}^{n\times n}$, $\Ind {\bf A}=k$,  $\rk {\bf A}^k=s$, and there exist $\coepr{\bf A}$ and $\coepl{\bf A}$. Then ${\coepr{\bf A}}=\left(a_{ij}^{\tiny\textcircled{\dag},r}\right)$ and ${\coepl{\bf A}}=\left(a_{ij}^{\tiny\textcircled{\dag},l}\right)$ possess the  determinantal representations, respectively,
 \begin{align}\label{eq:detrep_repcorep}
a_{ij}^{\tiny\textcircled{\dag},r}=&
 \frac{\sum\limits_{\alpha \in I_{s,n} {\left\{ {j}
\right\}}} {{\rm{rdet}} _{j} \left( {\left( { {\bf A}^{ k+1}\left({\bf A}^{k+1} \right)^{*}
} \right)_{j.} ({\hat {\bf a}}_{i\,.})} \right)_{\alpha} ^{\alpha} } }{{{\sum\limits_{\alpha \in I_{s,n}} {{\left|   {\bf A}^{k+1}\left({\bf A}^{k+1}\right)^{*}
  \right|_{\alpha} ^{\alpha}}}}
 }},\\\label{eq:detrep_lepcorep}
 a_{ij}^{\tiny\textcircled{\dag},l}=
& \frac{\sum\limits_{\beta \in J_{s,n} {\left\{ {i}
\right\}}} {{\rm{cdet}} _{i} \left( {\left( \left({\bf A}^{k+1} \right)^{*} {\bf A}^{ k+1}
\right)_{.i} ({\check {\bf a}}_{.j})} \right)_{\beta} ^{\beta} } }{{{\sum\limits_{\beta \in J_{s,n}} {{\left| \left({\bf A}^{k+1}\right)^{*} {\bf A}^{k+1}
  \right|_{\beta} ^{\beta}}}} }},
\end{align}
 where ${\hat {\bf a}}_{i\,.}$ is the $i$-th row of $\hat{{\bf A}}=
{\bf A}^{ k}({\bf A}^{ k+1})^{*}$ and ${\check {\bf a}}_{.j}$ is the $j$-th column of $\check{{\bf A}}=({\bf A}^{ k+1})^{*}
{\bf A}^{ k}$.
\end{thm}

\begin{cor}\cite{ky_cor}Let ${\bf A}\in  {\mathbb{H}}^{n\times n}_s$, $\Ind {\bf A}=1$, and there exist $\corer{\bf A}$ and $\corel{\bf A}$. Then ${\corer{\bf A}}=\left(a_{ij}^{\tiny\textcircled{\#},r}\right)$ and ${\corel{\bf A}}=\left(a_{ij}^{\tiny\textcircled{\#},l}\right)$ can be expressed  as follows

\begin{align}\label{eq:detrep_repcorep_sim}
a_{ij}^{\tiny\textcircled{\#},r}=&
 \frac{\sum\limits_{\alpha \in I_{s,n} {\left\{ {j}
\right\}}} {{\rm{rdet}} _{j} \left( {\left( { {\bf A}^{2}\left({\bf A}^{2} \right)^{*}
} \right)_{j.} ({\hat {\bf a}}_{i\,.})} \right)_{\alpha} ^{\alpha} } }{{{\sum\limits_{\alpha \in I_{s,n}} {{\left|   {\bf A}^{2}\left({\bf A}^{2}\right)^{*}
  \right|_{\alpha} ^{\alpha}}}}
 }},\\\label{eq:detrep_lepcorep_sim}
 a_{ij}^{\tiny\textcircled{\#},l}=
& \frac{\sum\limits_{\beta \in J_{s,n} {\left\{ {i}
\right\}}} {{\rm{cdet}} _{i} \left( {\left( \left({\bf A}^{2} \right)^{*} {\bf A}^{2}
\right)_{.i} ({\check {\bf a}}_{.j})} \right)_{\beta} ^{\beta} } }{{{\sum\limits_{\beta \in J_{s,n}} {{\left| \left({\bf A}^{2}\right)^{*} {\bf A}^{2}
  \right|_{\beta} ^{\beta}}}} }},
\end{align}
 where ${\hat {\bf a}}_{i\,.}$ is the $i$-th row of $\hat{{\bf A}}=
{\bf A}({\bf A}^{2})^{*}$ and ${\check {\bf a}}_{.j}$ is the $j$-th column of $\check{{\bf A}}=({\bf A}^{2})^{*}
{\bf A}$.
\end{cor}
The following corollary gives determinantal representations of the right and left, core and core-EP inverses  for complex matrices.
\begin{cor}\cite{ky_cor1}Suppose ${\bf A}\in  {\mathbb{C}}^{n\times n}_s$, $Ind {\bf A}=k$, and there exist
 ${\bf A}^{\tiny\textcircled{\dag}}=\left(a_{ij}^{\tiny\textcircled{\dag},r}\right)$ and ${\bf A}_{\tiny\textcircled{\dag}}=\left(a_{ij}^{\tiny\textcircled{\dag},l}\right)$. Then they have the following determinantal representations, respectively,
 \begin{align*}
a_{ij}^{\tiny\textcircled{\dag},r}=&
 \frac{\sum\limits_{\alpha \in I_{s,n} {\left\{ {j}
\right\}}} { \left| {\left( { {\bf A}^{ k+1}\left({\bf A}^{k+1} \right)^{*}
} \right)_{j.} ({\hat {\bf a}}_{i\,.})} \right|_{\alpha} ^{\alpha} } }{{{\sum\limits_{\alpha \in I_{s,n}} {{\left|   {\bf A}^{k+1}\left({\bf A}^{k+1}\right)^{*}
  \right|_{\alpha} ^{\alpha}}}}
 }},\\
 a_{ij}^{\tiny\textcircled{\dag},l}=
& \frac{\sum\limits_{\beta \in J_{s,n} {\left\{ {i}
\right\}}} { \left| {\left( \left({\bf A}^{k+1} \right)^{*} {\bf A}^{ k+1}
\right)_{.i} ({\check {\bf a}}_{.j})} \right|_{\beta} ^{\beta} } }{{{\sum\limits_{\beta \in J_{s,n}} {{\left| \left({\bf A}^{k+1}\right)^{*} {\bf A}^{k+1}
  \right|_{\beta} ^{\beta}}}} }},
\end{align*}
 where ${\hat {\bf a}}_{i\,.}$ is the $i$-th row of $\hat{{\bf A}}=
{\bf A}^{ k}({\bf A}^{ k+1})^{*}$ and ${\check {\bf a}}_{.j}$ is the $j$-th column of $\check{{\bf A}}=({\bf A}^{ k+1})^{*}
{\bf A}^{ k}$.

If $Ind {\bf A}=1$, then
 ${\bf A}^{\tiny\textcircled{\#}}=\left(a_{ij}^{\tiny\textcircled{\#},r}\right)$ and ${\bf A}_{\tiny\textcircled{\#}}=\left(a_{ij}^{\tiny\textcircled{\#},l}\right)$ have the following determinantal representations, respectively,
 \begin{align*}
a_{ij}^{\tiny\textcircled{\#},r}=&
 \frac{\sum\limits_{\alpha \in I_{s,n} {\left\{ {j}
\right\}}} { \left| {\left( { {\bf A}^{ 2}\left({\bf A}^{2} \right)^{*}
} \right)_{j.} ({\hat {\bf a}}_{i\,.})} \right|_{\alpha} ^{\alpha} } }{{{\sum\limits_{\alpha \in I_{s,n}} {{\left|   {\bf A}^{2}\left({\bf A}^{2}\right)^{*}
  \right|_{\alpha} ^{\alpha}}}}
 }},\\
 a_{ij}^{\tiny\textcircled{\#},l}=
& \frac{\sum\limits_{\beta \in J_{s,n} {\left\{ {i}
\right\}}} { \left| {\left( \left({\bf A}^{2} \right)^{*} {\bf A}^{2}
\right)_{.i} ({\check {\bf a}}_{.j})} \right|_{\beta} ^{\beta} } }{{{\sum\limits_{\beta \in J_{s,n}} {{\left| \left({\bf A}^{2}\right)^{*} {\bf A}^{2}
  \right|_{\beta} ^{\beta}}}} }},
\end{align*}
 where ${\hat {\bf a}}_{i\,.}$ is the $i$-th row of $\hat{\bf A}=
{\bf A}({\bf A}^{ 2})^{*}$ and ${\check {\bf a}}_{.j}$ is the $j$-th column of $\check{{\bf A}}=({\bf A}^{2})^{*}
{\bf A}$.

\end{cor}

\section{Concepts of quaternion W-weighted
core-EP  inverses and their determinantal representations}

 The concept of the W-weighted
core-EP  inverse in complex matrices was introduced by Ferreyra et al. \cite{fer2} that can be expended to quaternion matrices as follows.
\begin{defn}\label{def_wrcep}Suppose ${\bf A}\in  {\mathbb{H}}^{m \times n}$,  ${\bf W}\in  {\mathbb{H}}^{n \times m}$, and $k=\max\{\Ind ({\bf W}{\bf A}),$ $\Ind ({\bf A}{\bf W})\}$.
The right W-weighted core-EP inverse of $ {\bf A}$  is the unique solution to the system \begin{align*}
{\bf W}{\bf A}{\bf W}{\bf X}=\left({\bf W}{\bf A}\right)^k\left[{\left({\bf W}{\bf A}\right)^k}\right]^{\dag},~ {\text{and}}~ \mathcal{C}_r({\bf X})\subseteq\mathcal{C}_r\left({\left({\bf A}{\bf W}\right)^k}\right). \end{align*}
It is denoted  ${\bf A}^{\tiny\textcircled{\dag},W,r}$.
\end{defn}
Due to \cite{pras}, the right weighted core-EP inverse over the quaternion skew field can be determined as follows.
\begin{thm}Let ${\bf A},~{\bf X} \in{\mathbb{H}}^{m\times n}$, ${\bf W} \in{\mathbb{H}}^{n\times m}$, and $k=\max\{\Ind ({\bf W}{\bf A}),$ $\Ind ({\bf A}{\bf W})\}$. The following statements are
equivalent:
\begin{itemize}
  \item [(i)]  ${\bf X}$ is the right weighted core-EP inverse of $ {\bf A}$.
  \item [(ii)] ${\bf X}{\bf W}\left({\bf A}{\bf W}\right)^{k+1}= \left({\bf A}{\bf W}\right)^k,~{\bf A}{\bf W}{\bf X} {\bf
W}{\bf X}   = {\bf X},~ {\text{and}}~ \left({\bf W}{\bf A}{\bf W}{\bf X}\right)^*={\bf W}{\bf A}{\bf W}{\bf X}$.
\item [(iii)] \begin{align}\label{eq:rep rwcep}{\bf X}={\bf A} \left[\left({\bf W}{\bf A}\right)^{\tiny\textcircled{\dag}}\right]^2.\end{align}
\end{itemize}
\end{thm}

We propose to introduce a left weighted core-EP inverse as well.
\begin{defn}\label{def_wlcep}Suppose ${\bf A}\in  {\mathbb{H}}^{m \times n}$,  ${\bf W}\in  {\mathbb{H}}^{n \times m}$, and $k=\max\{\Ind ({\bf W}{\bf A}),$ $\Ind ({\bf A}{\bf W})\}$.
The left W-weighted core-EP inverse of $ {\bf A}$  is the unique solution to the system \begin{align}\label{eq:def_wlcep}
{\bf X}{\bf W}{\bf A}{\bf W}=\left[{\left({\bf A}{\bf W}\right)^k}\right]^{\dag}\left({\bf A}{\bf W}\right)^k,~{\text {and}}~ \mathcal{R}_l({\bf X})\subseteq\mathcal{R}_l\left({\left({\bf W}{\bf A}\right)^k}\right). \end{align}
It is denoted  ${\bf A}^{\tiny\textcircled{\dag},W,l}$.
\end{defn}

\begin{thm}\label{thm_wlcep}Let ${\bf A},~{\bf X} \in{\mathbb{H}}^{m\times n}$, ${\bf W} \in{\mathbb{H}}^{n\times m}$, and $k=\max\{\Ind ({\bf W}{\bf A}),$ $\Ind ({\bf A}{\bf W})\}$. The following statements are
equivalent:
\begin{itemize}
  \item [(i)]~$ {\bf X}=\left[\left({\bf A}{\bf W}\right)_{\tiny\textcircled{\dag}}\right]^2{\bf A}.$
   \item [(ii)]
   ${\bf X}$ is the left weighted core-EP inverse of $ {\bf A}$.
  \item  [(iii)]
  ${\bf X}$ is the unique solution to the three equations:
   \begin{align}\label{eq:wlcep}\left({\bf W}{\bf A}\right)^{k+1}{\bf W}{\bf X}= \left({\bf W}{\bf A}\right)^k,~{\bf X}{\bf W}{\bf X} {\bf
W}{\bf A}   = {\bf X},~ \text{and}~ \left({\bf X{\bf W}{\bf A}{\bf W}}\right)^*={\bf X}{\bf W}{\bf A}{\bf W}.\end{align}

\end{itemize}
\end{thm}
\begin{proof}$(i)\mapsto (ii)$
We show that ${\bf X}=\left[\left({\bf A}{\bf W}\right)_{\tiny\textcircled{\dag}}\right]^2{\bf A} $ satisfies the condition (\ref{eq:def_wlcep}). Indeed,
\begin{align*}
{\bf X}{\bf W}{\bf A}{\bf W}=&\left[\left({\bf A}{\bf W}\right)_{\tiny\textcircled{\dag}}\right]^2{\bf A}{\bf W}{\bf A}{\bf W}=\left({\bf A}{\bf W}\right)_{\tiny\textcircled{\dag}}{\bf A}{\bf W}=\\=&\left[{\left({\bf A}{\bf W}\right)^{k}}\right]^{\dag}\left({\bf A}{\bf W}\right)^{k}\left({\bf A}{\bf W}\right)^d{\bf A}{\bf W}=\left[{\left({\bf A}{\bf W}\right)^k}\right]^{\dag}\left({\bf A}{\bf W}\right)^k,\\
\left[\left({\bf A}{\bf W}\right)_{\tiny\textcircled{\dag}}\right]^2{\bf A}=&\left[\left({\bf A}{\bf W}\right)_{\tiny\textcircled{\dag}}\right]^3{\bf A}{\bf W}{\bf A}=\ldots=\left[\left({\bf A}{\bf W}\right)_{\tiny\textcircled{\dag}}\right]^{k+2}\left({\bf A}{\bf W}\right)^k{\bf A}=\\=&\left[\left({\bf A}{\bf W}\right)_{\tiny\textcircled{\dag}}\right]^{k+2}{\bf A}\left({\bf W}{\bf A}\right)^k, {\text {i.e.}}\\
\mathcal{R}_l &(\left[\left({\bf A}{\bf W}\right)_{\tiny\textcircled{\dag}}\right]^2{\bf A})\subseteq\mathcal{R}_l\left({\left({\bf W}{\bf A}\right)^k}\right). \end{align*}

$(ii)\mapsto (iii)$ Now, we prove that ${\bf X}=\left[\left({\bf A}{\bf W}\right)_{\tiny\textcircled{\dag}}\right]^2{\bf A} $ satisfies the equations in (\ref{eq:wlcep}).

Since $\left({\bf W}{\bf A}\right)^{k+1}{\bf W}={\bf W}\left({\bf A}{\bf W}\right)^{k+1}$ and due to Theorem \ref{th:lcep}, 
$$\left({\bf A}{\bf W}\right)_{\tiny\textcircled{\dag}}=\left(\left({\bf A}{\bf W}\right)^k\right)^{\dag}\left({\bf A}{\bf W}\right)^k\left({\bf A}{\bf W}\right)^d,$$
 then
\begin{align*}
\left({\bf W}{\bf A}\right)^{k+1}{\bf W}\left[\left({\bf A}{\bf W}\right)_{\tiny\textcircled{\dag}}\right]^2{\bf A}=&{\bf W}\left[\left({\bf A}{\bf W}\right)^{k+1}\left({\bf A}{\bf W}\right)_{\tiny\textcircled{\dag}}\right]\left({\bf W}{\bf A}\right)_{\tiny\textcircled{\dag}}{\bf A}=\\
=&{\bf W}\left({\bf A}{\bf W}\right)^{k}\left(\left({\bf A}{\bf W}\right)^k\right)^{\dag}\left({\bf A}{\bf W}\right)^k\left({\bf A}{\bf W}\right)^d{\bf A}=\\
=&{\bf W}\left({\bf A}{\bf W}\right)^{k}\left({\bf A}{\bf W}\right)^d{\bf A}=\\
=&\left({\bf W}{\bf A}\right)^{k+1}\left({\bf W}{\bf A}\right)^d=
\\=&
 \left({\bf W}{\bf A}\right)^k.
\end{align*}
By Theorem \ref{th:lcep}, taking into account $\left[\left({\bf A}{\bf W}\right)_{\tiny\textcircled{\dag}}\right]^2{\bf A} {\bf
W}=\left({\bf A}{\bf W}\right)_{\tiny\textcircled{\dag}}$, we have
\begin{align*}
\left[\left({\bf A}{\bf W}\right)_{\tiny\textcircled{\dag}}\right]^2{\bf A}{\bf W}\left(\left[\left({\bf A}{\bf W}\right)_{\tiny\textcircled{\dag}}\right]^2{\bf A} {\bf
W}\right){\bf A}   =&\left[\left({\bf A}{\bf W}\right)_{\tiny\textcircled{\dag}}\right]^2{\bf A}{\bf W}\left({\bf A}{\bf W}\right)_{\tiny\textcircled{\dag}}{\bf A} =\\=&
 \left[\left({\bf A}{\bf W}\right)_{\tiny\textcircled{\dag}}\right]^2{\bf A}.
\end{align*}
Finally,
\begin{align*}
\left(\left[\left({\bf A}{\bf W}\right)_{\tiny\textcircled{\dag}}\right]^2{\bf A}{\bf W}{\bf A}{\bf W}\right)^*   =&\left(\left({\bf A}{\bf W}\right)_{\tiny\textcircled{\dag}}{\bf A}{\bf W}\right)^*=\left({\bf A}{\bf W}\right)_{\tiny\textcircled{\dag}}{\bf A}{\bf W}.
\end{align*}
Now, we  prove  the uniqueness of  ${\bf X}$. Let
 \begin{align*}\left({\bf W}{\bf A}\right)^{k+1}{\bf W}{\bf X}= &\left({\bf W}{\bf A}\right)^k,~{\bf X}{\bf W}{\bf X} {\bf
W}{\bf A}   = {\bf X},~  \left({\bf X}{\bf W}{\bf A}{\bf W}\right)^*={\bf X}{\bf W}{\bf A}{\bf W},\\ ~& {\text {and}}~{\bf X}=\left[\left({\bf A}{\bf W}\right)_{\tiny\textcircled{\dag}}\right]^2{\bf A}.\end{align*} Suppose there also exists the left weighted core-EP inverse ${\bf Y}$ such that
 \begin{align*}\left({\bf W}{\bf A}\right)^{k+1}{\bf W}{\bf Y}= \left({\bf W}{\bf A}\right)^k,~{\bf Y}{\bf W}{\bf Y} {\bf
W}{\bf A}   = {\bf Y},~ {\text {and}}~ \left({\bf Y}{\bf W}{\bf A}{\bf W}\right)^*={\bf Y}{\bf W}{\bf A}{\bf W}.\end{align*} We show that ${\bf Y}={\bf X}=\left[\left({\bf A}{\bf W}\right)_{\tiny\textcircled{\dag}}\right]^2{\bf A}$.
So,
 \begin{align*} {\bf Y}=&{\bf Y}{\bf W}{\bf Y} {\bf
W}{\bf A}   ={\bf Y}\left({\bf W}{\bf Y}\right)^2\left( {\bf
W}{\bf A}\right)^2={\bf Y}\left({\bf W}{\bf Y}\right)^k\left( {\bf
W}{\bf A}\right)^k=\\=&{\bf Y}\left({\bf W}{\bf Y}\right)^k\left({\bf W}{\bf A}\right)^{k+1}{\bf W}{\bf X}=\left({\bf Y}{\bf W}\right)^{k+1}\left({\bf A}{\bf W}\right)^{k+1}{\bf X}=\left({\bf Y}{\bf W}{\bf A}{\bf W}\right)^{k+1}{\bf X}=\\=&
\left(\left[{\left({\bf A}{\bf W}\right)^k}\right]^{\dag}\left({\bf A}{\bf W}\right)^k\right)^{k+1}{\bf X}=\left[{\left({\bf A}{\bf W}\right)^{2k+1}}\right]^{\dag}\left({\bf A}{\bf W}\right)^{2k+1}\left[\left({\bf A}{\bf W}\right)_{\tiny\textcircled{\dag}}\right]^2{\bf A}.\end{align*} By Theorem \ref{th1:lcep}, $\left({\bf A}{\bf W}\right)_{\tiny\textcircled{\dag}}=\left[{\left({\bf A}{\bf W}\right)^{2k+1}}\right]^{\dag}\left({\bf A}{\bf W}\right)^{2k+1}\left({\bf A}{\bf W}\right)^d.$
So,
 \begin{align*} {\bf Y}=&\left[{\left({\bf A}{\bf W}\right)^{2k+1}}\right]^{\dag}\left({\bf A}{\bf W}\right)^{2k+1}\left[{\left({\bf A}{\bf W}\right)^{2k+1}}\right]^{\dag}\left({\bf A}{\bf W}\right)^{2k+1}\left({\bf A}{\bf W}\right)^d\left({\bf A}{\bf W}\right)_{\tiny\textcircled{\dag}}{\bf A}=\\=&\left[{\left({\bf A}{\bf W}\right)^{2k+1}}\right]^{\dag}\left({\bf A}{\bf W}\right)^{2k+1}\left({\bf A}{\bf W}\right)^d\left({\bf A}{\bf W}\right)_{\tiny\textcircled{\dag}}{\bf A}=\left[\left({\bf A}{\bf W}\right)_{\tiny\textcircled{\dag}}\right]^2{\bf A}.
\end{align*}
Finally, from the uniqueness of  ${\bf X}$ follows $(iii)\mapsto (i)$.
\end{proof}

Now, we give determinantal representations of the quaternion W-weighted
core-EP  inverses.

\begin{thm}\label{th:drrcep}
Suppose ${\bf A}\in  {\mathbb{H}}^{m \times n}$,  ${\bf W}\in  {\mathbb{H}}^{n \times m}$, and $k=\max\{\Ind ({\bf W}{\bf A}),$ $\Ind ({\bf A}{\bf W})\}$,  $\rk ({\bf W}{\bf A})^k=s$.
Denote ${\bf W}{\bf A}:={\bf U}=\left(u_{ij}\right)\in  {\mathbb{H}}^{n \times n}$. Then
 the right weighted core-EP inverse
 ${\bf A}^{\tiny\textcircled{\dag},W,r}=\left(a_{ij}^{\tiny\textcircled{\dag},W,r}\right)\in  {\mathbb{H}}^{m\times n}$
 possess the  determinantal representations
 \begin{align}\label{eq:detrep_rwepcorep}
a_{ij}^{\tiny\textcircled{\dag},W,r}=&
 \frac{\sum\limits_{\alpha \in I_{s,n} {\left\{ {j}
\right\}}} {{\rm{rdet}} _{j} \left( {\left( { {\bf U}^{ k+1}\left({\bf U}^{k+1} \right)^{*}
} \right)_{j.} (\widetilde{\phi}_{i.})} \right)_{\alpha} ^{\alpha} } }{{\left({\sum\limits_{\alpha \in I_{s,n}} {{\left|   {\bf U}^{k+1}\left({\bf U}^{k+1}\right)^{*}
  \right|_{\alpha} ^{\alpha}}}}\right)^2
 }},
\end{align}
 where ${\widetilde {\bf \phi}}_{i.}$
 is  the $i$-th row of $\widetilde{{\bf \Phi}}={\bf \Phi}{\bf U}^{ k}\left({\bf U}^{ k+1}\right)^{*}$. The matrix is determined ${\bf \Phi}=\left(\phi_{if}\right)$ as follows

 \begin{align*}\phi_{if}=\sum\limits_{\alpha \in I_{s,n} {\left\{ {f}
\right\}}} {{\rm{rdet}} _{f} \left( {\left( { {\bf U}^{ k+1}\left({\bf U}^{k+1} \right)^{*}
} \right)_{f.} ({\widetilde {\bf u}}_{i.})} \right)_{\alpha} ^{\alpha} }, \end{align*}
 where
 $\widetilde{{\bf u}}_{i\,.}$ is the $i$-th row of $\widetilde{{\bf U}}=
{\bf A}{\bf U}^{ k}\left({\bf U}^{ k+1}\right)^{*}$.
\end{thm}
\begin{proof}
Taking into account (\ref{eq:rep rwcep}) for ${\bf A}^{\tiny\textcircled{\dag},W,r}$, we get
\begin{align}\label{eq:rwcep_s}
a_{ij}^{\tiny\textcircled{\dag},W,r} =&  \sum\limits_{l = 1}^{n} \sum\limits_{f = 1}^{n} {a}_{il}u_{lf}^{\tiny\textcircled{\dag},r}u_{fj}^{\tiny\textcircled{\dag},r}. \end{align}
Using the determinantal representation (\ref{eq:detrep_repcorep}) of ${\bf U}^{\tiny\textcircled{\dag}}$ gives
\begin{align*}\phi_{if}= &\sum\limits_{l = 1}^{n} {a}_{il}u_{lf}^{\tiny\textcircled{\dag},r}= \frac{\sum\limits_{l = 1}^{n} {a}_{il}
\sum\limits_{\alpha \in I_{s,n} {\left\{ {f}
\right\}}} {{\rm{rdet}} _{f} \left( {\left( { {\bf U}^{ k+1}\left({\bf U}^{k+1} \right)^{*}
} \right)_{f.} ({\hat {\bf u}}_{l.})} \right)_{\alpha} ^{\alpha} } }{{{\sum\limits_{\alpha \in I_{s,n}} {{\left|   {\bf U}^{k+1}\left({\bf U}^{k+1}\right)^{*}
  \right|_{\alpha} ^{\alpha}}}} }}=\\=&\frac{
\sum\limits_{\alpha \in I_{s,n} {\left\{ {f}
\right\}}} {{\rm{rdet}} _{f} \left( {\left( { {\bf U}^{ k+1}\left({\bf U}^{k+1} \right)^{*}
} \right)_{f.} ({\widetilde {\bf u}}_{i.})} \right)_{\alpha} ^{\alpha} } }{{{\sum\limits_{\alpha \in I_{s,n}} {{\left|   {\bf U}^{k+1}\left({\bf U}^{k+1}\right)^{*}
  \right|_{\alpha} ^{\alpha}}}} }},
 \end{align*}
where ${\hat {\bf u}}_{l.}$ is the $l$-th row of   $\widehat{{\bf U}}={\bf U}^{k}\left({\bf U}^{k+1}\right)^{*}$ and ${\widetilde {\bf u}}_{i.}$ is the $i$-th row of   $\widetilde{{\bf U}}={\bf A}{\bf U}^{k}\left({\bf U}^{k+1}\right)^{*}$.

 Denote
 \begin{align*}\phi_{if}=\sum\limits_{\alpha \in I_{s,n} {\left\{ {f}
\right\}}} {{\rm{rdet}} _{f} \left( {\left( { {\bf U}^{ k+1}\left({\bf U}^{k+1} \right)^{*}
} \right)_{f.} ({\widetilde {\bf u}}_{i.})} \right)_{\alpha} ^{\alpha} }. \end{align*}
 Substituting $\phi_{if}$ into  (\ref{eq:rwcep_s}) gives
 \begin{align*}
 a_{ij}^{\tiny\textcircled{\dag},W,r} =&\sum\limits_{l = 1}^{n} \phi_{if}u_{fj}^{\tiny\textcircled{\dag},r}= \frac{ \sum\limits_{f = 1}^{n}\phi_{if}
\sum\limits_{\alpha \in I_{s,n} {\left\{ {f}
\right\}}} {{\rm{rdet}} _{j} \left( {\left( { {\bf U}^{ k+1}\left({\bf U}^{k+1} \right)^{*}
} \right)_{j.} ({\widehat {\bf u}}_{f.})} \right)_{\alpha} ^{\alpha} } }{\left({{\sum\limits_{\alpha \in I_{s,n}} {{\left|   {\bf U}^{k+1}\left({\bf U}^{k+1}\right)^{*}
  \right|_{\alpha} ^{\alpha}}}} }\right)^2}.
  \end{align*}
Putting $\sum\limits_{f = 1}^{n}\phi_{if}{\widehat {\bf u}}_{f.}={\widetilde {\bf \phi}}_{i.}$
 as  the $i$-th row of $\widetilde{{\bf \Phi}}={\bf \Phi}{\bf U}^{ k}\left({\bf U}^{ k+1}\right)^{*}$ yields (\ref{eq:detrep_rwepcorep}).
\end{proof}

Similarly,  it can be proved theorem on the determinantal representation of the quaternion W-weighted left core-EP  inverse.

\begin{thm}\label{th:drlcep}
Suppose ${\bf A}\in  {\mathbb{H}}^{m \times n}$,  ${\bf W}\in  {\mathbb{H}}^{n \times m}$, and $k=\max\{\Ind ({\bf W}{\bf A}),$ $\Ind ({\bf A}{\bf W})\}$,  $\rk ({\bf A}{\bf W})^k=s$.
Denote ${\bf A}{\bf W}:={\bf V}=\left(v_{ij}\right)\in  {\mathbb{H}}^{m \times m}$. Then
 the right weighted core-EP inverse
 ${\bf A}^{\tiny\textcircled{\dag},W,l}=\left(a_{ij}^{\tiny\textcircled{\dag},W,l}\right)\in  {\mathbb{H}}^{n\times m}$
 possess the  determinantal representations
 \begin{align*}
a_{ij}^{\tiny\textcircled{\dag},W,l}=&
 \frac{\sum\limits_{\beta \in J_{s,m} {\left\{ {i}
\right\}}} {{\rm{cdet}} _{i} \left( {\left( {\left({\bf V}^{k+1} \right)^{*} {\bf V}^{ k+1}
} \right)_{.i} (\widetilde{\psi}_{.j})} \right)_{\beta} ^{\beta} } }{{\left({\sum\limits_{\beta \in J_{s,m}} {{\left|  \left({\bf V}^{k+1}\right)^{*} {\bf V}^{k+1}
  \right|_{\beta} ^{\beta}}}}\right)^2
 }},
\end{align*}
 where ${\widetilde {\bf \psi}}_{.j}$
 is  the $j$-th column of $\widetilde{{\bf \Psi}}=\left({\bf V}^{ k+1}\right)^{*}{\bf V}^{ k}{\bf \Psi}$. The matrix is determined ${\bf \Psi}=\left(\psi_{lj}\right)\in {\mathbb{H}}^{m \times n}$ as follows
 \begin{align*}\psi_{lj}=\sum\limits_{\beta \in J_{s,m} {\left\{ {l}
\right\}}} {{\rm{cdet}} _{l} \left( {\left( { \left({\bf V}^{k+1} \right)^{*}{\bf V}^{ k+1}
} \right)_{.l} ({\widetilde {\bf v}}_{.j})} \right)_{\beta} ^{\beta} }, \end{align*}
 where
 $\widetilde{{\bf v}}_{.j}$ is the $j$-th column of $\widetilde{{\bf V}}=
\left({\bf V}^{ k+1}\right)^{*}{\bf V}^{ k}{\bf A}$.
\end{thm}

It's evident that  the following corollaries can be get in the complex case.
\begin{cor}
Suppose ${\bf A}\in  {\mathbb{C}}^{m \times n}$,  ${\bf W}\in  {\mathbb{C}}^{n \times m}$, and $k=\max\{\Ind ({\bf W}{\bf A}),$ $\Ind ({\bf A}{\bf W})\}$,  $\rk ({\bf W}{\bf A})^k=s$.
Denote ${\bf W}{\bf A}:={\bf U}=\left(u_{ij}\right)\in  {\mathbb{C}}^{n \times n}$. Then
 the right weighted core-EP inverse
 ${\bf A}^{\tiny\textcircled{\dag},W,r}=\left(a_{ij}^{\tiny\textcircled{\dag},W,r}\right)\in  {\mathbb{C}}^{m\times n}$
 possess the  determinantal representations
 \begin{align*}
a_{ij}^{\tiny\textcircled{\dag},W,r}=&
 \frac{\sum\limits_{\alpha \in I_{s,n} {\left\{ {j}
\right\}}} { \left| {\left( { {\bf U}^{ k+1}\left({\bf U}^{k+1} \right)^{*}
} \right)_{j.} (\widetilde{\phi}_{i.})} \right|_{\alpha} ^{\alpha} } }{{\left({\sum\limits_{\alpha \in I_{s,n}} {{\left|   {\bf U}^{k+1}\left({\bf U}^{k+1}\right)^{*}
  \right|_{\alpha} ^{\alpha}}}}\right)^2
 }},
\end{align*}
 where ${\widetilde {\bf \phi}}_{i.}$
 is  the $i$-th row of $\widetilde{{\bf \Phi}}={\bf \Phi}{\bf U}^{ k}\left({\bf U}^{ k+1}\right)^{*}$. The matrix is determined ${\bf \Phi}=\left(\phi_{if}\right)$ as follows

 \begin{align*}\phi_{if}=\sum\limits_{\alpha \in I_{s,n} {\left\{ {f}
\right\}}} { \left| {\left( { {\bf U}^{ k+1}\left({\bf U}^{k+1} \right)^{*}
} \right)_{f.} ({\widetilde {\bf u}}_{i.})} \right|_{\alpha} ^{\alpha} }, \end{align*}
 where
 $\widetilde{{\bf u}}_{i\,.}$ is the $i$-th row of $\widetilde{{\bf U}}=
{\bf A}{\bf U}^{ k}\left({\bf U}^{ k+1}\right)^{*}$.
\end{cor}

\begin{cor}
Suppose ${\bf A}\in  {\mathbb{C}}^{m \times n}$,  ${\bf W}\in  {\mathbb{C}}^{n \times m}$, and $k=\max\{\Ind ({\bf W}{\bf A}),$ $\Ind ({\bf A}{\bf W})\}$,  $\rk ({\bf A}{\bf W})^k=s$.
Denote ${\bf A}{\bf W}:={\bf V}=\left(v_{ij}\right)\in  {\mathbb{C}}^{m \times m}$. Then
 the right weighted core-EP inverse
 ${\bf A}^{\tiny\textcircled{\dag},W,l}=\left(a_{ij}^{\tiny\textcircled{\dag},W,l}\right)\in  {\mathbb{H}}^{n\times m}$
 possess the  determinantal representations
 \begin{align*}
a_{ij}^{\tiny\textcircled{\dag},W,l}=&
 \frac{\sum\limits_{\beta \in J_{s,m} {\left\{ {i}
\right\}}} { \left| {\left( {\left({\bf V}^{k+1} \right)^{*} {\bf V}^{ k+1}
} \right)_{.i} (\widetilde{\psi}_{.j})} \right|_{\beta} ^{\beta} } }{{\left({\sum\limits_{\beta \in J_{s,m}} {{\left|  \left({\bf V}^{k+1}\right)^{*} {\bf V}^{k+1}
  \right|_{\beta} ^{\beta}}}}\right)^2
 }},
\end{align*}
 where ${\widetilde {\bf \psi}}_{.j}$
 is  the $j$-th column of $\widetilde{{\bf \Psi}}=\left({\bf V}^{ k+1}\right)^{*}{\bf V}^{ k}{\bf \Psi}$. The matrix is determined ${\bf \Psi}=\left(\psi_{lj}\right)\in {\mathbb{H}}^{m \times n}$ as follows
 \begin{align*}\psi_{lj}=\sum\limits_{\beta \in J_{s,m} {\left\{ {l}
\right\}}} { \left| {\left( { \left({\bf V}^{k+1} \right)^{*}{\bf V}^{ k+1}
} \right)_{.l} ({\widetilde {\bf v}}_{.j})} \right|_{\beta} ^{\beta} }, \end{align*}
 where
 $\widetilde{{\bf v}}_{.j}$ is the $j$-th column of $\widetilde{{\bf V}}=
\left({\bf V}^{ k+1}\right)^{*}{\bf V}^{ k}{\bf A}$.
\end{cor}

\section{The W-weighted  DMP and MPD inverses and their determinantal representations.}
 The concept of the DMP  inverse in complex matrices was introduced in \cite{mal1} by S. Malik and N. Thome that can be expended to quaternion matrices as follows.
\begin{defn}Suppose ${\bf A}\in  {\mathbb{H}}^{n\times n}$ and $\Ind {\bf A}=k$.
A matrix ${\bf X}\in  {\mathbb{H}}^{n\times n}$ is said to be the DMP inverse of $ {\bf A}$
if it satisfies the conditions
 \begin{align}\label{eq:def_dmp}
{\bf X}{\bf A}{\bf X}={\bf X},~{\bf X}{\bf A}={\bf A}^d{\bf A},~and~ {\bf A}^k{\bf X}={\bf A}^k{\bf A}^{\dag}. \end{align}
It is denoted $ {\bf A}^{d,{\dag}}$.
\end{defn}Similar as for complex matrices  \cite{mal1}, if a quaternion matrix  satisfies
 the system of equations (\ref{eq:def_dmp}), then it is unique and has the  representation,
$$
{\bf A}^{d,{\dag}}={\bf A}^{d}{\bf A}{\bf A}^{\dag}. $$

In \cite{ky_cor}, we also introduce the MPD inverse.
\begin{defn}\label{eq:def_mpd}
Suppose ${\bf A}\in  {\mathbb{H}}^{n\times n}$ and $\Ind {\bf A}=k$.
A matrix ${\bf X}\in  {\mathbb{H}}^{n\times n}$ is said to be the MPD inverse of $ {\bf A}$
if it satisfies the conditions
 \begin{align*}
{\bf X}{\bf A}{\bf X}={\bf X},~{\bf A}{\bf X}={\bf A}{\bf A}^d,~and~ {\bf X}{\bf A}^k={\bf A}^{\dag}{\bf A}^k. \end{align*}
It is denoted $ {\bf A}^{{\dag},d}$.
\end{defn}
It is not difficult to show that  $ {\bf A}^{{\dag},d}$ is unique and it can be represented as
$
{\bf A}^{{\dag},d}={\bf A}^{\dag}{\bf A}{\bf A}^{d}.$

In \cite{ky_cor}, we gave the determinantal representations of the  DMP  and MPD inverses over the quaternion skew field.

Recently in \cite{meng}, the definition of the DMP inverse of a square matrix with complex elements was extended to rectangular matrices. We  extend it over  the quaternion skew field.
\begin{defn}
Let ${\bf A}\in  {\mathbb{H}}^{m\times n}$ and ${\bf W}\in  {\mathbb{H}}^{n\times m}$ be a nonzero matrix. The W-weighted DMP (WDMP) inverse of ${\bf A}$ with respect to  ${\bf W}$ is defined as
\begin{align*}
{\bf A}^{d,\dag,W}={{\bf W}\bf A}^{d,W}{\bf W}{\bf A}{\bf A}^{\dag}.
\end{align*}
\end{defn}
Similarly to complex matrices  can be proved the next lemma.
\begin{lem}
Let ${\bf A}\in  {\mathbb{H}}^{n\times n}$ and  $k=\max\{\Ind ({\bf W}{\bf A}),$ $\Ind ({\bf A}{\bf W})\}$. The matrix  $
{\bf X}={\bf A}^{d,\dag,W}$ is the unique matrix that
satisfies the following system of equations
\begin{align}\label{eq:prop_wdmp}
{\bf X}{\bf A}{\bf X}={\bf X},~{\bf X}{\bf A}={\bf W}{\bf A}^{d,W}{\bf W}{\bf A},~and~ ({\bf W}{\bf A})^{k+1}{\bf X}=({\bf W}{\bf A})^{k+1}{\bf A}^{\dag}. \end{align}
\end{lem}
We propose to introduce the weighted MPD inverse as well.
\begin{lem}
Let ${\bf A}\in  {\mathbb{H}}^{n\times n}$ and  $k=\max\{\Ind ({\bf W}{\bf A}),$ $\Ind ({\bf A}{\bf W})\}$. Then the matrix  $
{\bf X}={\bf A}^{\dag}{\bf A}{\bf W}{\bf A}^{d,W}{\bf W}$ is the unique solution to the equations
\begin{align}\label{eq:wmpd}
{\bf X}{\bf A}{\bf X}={\bf X},~{\bf A}{\bf X}={\bf A}{\bf W}{\bf A}^{d,W}{\bf W},~and~ {\bf X}({\bf A}{\bf W})^{k+1}={\bf A}^{\dag}({\bf A}{\bf W})^{k+1}. \end{align}
\end{lem}
\begin{proof}
From Definitions \ref{def:wdr}, \ref{def:mp},   and taking into account (\ref{eq:WDr_dr}), it follows
\begin{align*}
{\bf A}^{\dag}{\bf A}{\bf W}{\bf A}^{d,W}{\bf W}({\bf A}{\bf A}^{\dag}{\bf A}){\bf W}{\bf A}^{d,W}{\bf W}=&{\bf A}^{\dag}{\bf A}{\bf W}({\bf A}^{d,W}{\bf W}{\bf A}{\bf W}{\bf A}^{d,W}){\bf W}=\\=&{\bf A}^{\dag}{\bf A}{\bf W}{\bf A}^{d,W}{\bf W},\\
{\bf A}{\bf A}^{\dag}{\bf A}{\bf W}{\bf A}^{d,W}{\bf W}=&{\bf A}{\bf W}{\bf A}^{d,W}{\bf W},\\
 {\bf A}^{\dag}{\bf A}{\bf W}{\bf A}^{d,W}{\bf W}({\bf A}{\bf W})^{k+1}=& {\bf A}^{\dag}{\bf A}{\bf W}\left(({\bf A}{\bf W})^{d}\right)^2{\bf A}{\bf W}({\bf A}{\bf W})^{k+1}=\\=&
 {\bf A}^{\dag}{\bf A}{\bf W}({\bf A}{\bf W})^{d}{\bf A}{\bf W}({\bf A}{\bf W})^d({\bf A}{\bf W})^{k+1}=\\=&{\bf A}^{\dag}({\bf A}{\bf W})^{k+1}.
 \end{align*}
It means that $
{\bf X}={\bf A}^{\dag}{\bf A}{\bf W}{\bf A}^{d,W}{\bf W}$ is the solution to the equations
(\ref{eq:wmpd}).

To prove uniqueness, suppose both ${\bf X}_1$ and ${\bf X}_2$ are two solutions to (\ref{eq:wmpd}). Using repeated applications of the  equations in (\ref{eq:wmpd}) and in Definition \ref{def:wdr}, we obtain
\begin{align*}
{\bf X}_1=&{\bf X}_1{\bf A}{\bf X}_1={\bf X}_1{\bf A}{\bf W}{\bf A}^{d,W}{\bf W}={\bf X}_1\left({\bf A}{\bf W}\right)^2\left({\bf A}^{d,W}{\bf W}\right)^2=\ldots=\\
=&{\bf X}_1\left({\bf A}{\bf W}\right)^{k+1}\left({\bf A}^{d,W}{\bf W}\right)^{k+1}=
{\bf A}^{\dag}\left({\bf A}{\bf W}\right)^{k+1}\left({\bf A}^{d,W}{\bf W}\right)^{k+1}=\\=&{\bf X}_2\left({\bf A}{\bf W}\right)^{k+1}\left({\bf A}^{d,W}{\bf W}\right)^{k+1}={\bf X}_2{\bf A}{\bf W}{\bf A}^{d,W}{\bf W}={\bf X}_2{\bf A}{\bf X}_2={\bf X}_2.
 \end{align*}
It completes the proof.
\end{proof}
\begin{defn}
Let ${\bf A}\in  {\mathbb{H}}^{m\times n}$ and ${\bf W}\in  {\mathbb{H}}^{n\times m}$ be a nonzero matrix. The W-weighted MPD (WMPD) inverse of ${\bf A}$ with respect to  ${\bf W}$ is defined as
\begin{align*}
{\bf A}^{\dag,d,W}={\bf A}^{\dag}{\bf A}{\bf W}{\bf A}^{d,W}{\bf W}.
\end{align*}
\end{defn}
Now, we give determinantal representations of the  WDMP  inverse.
We have two cases due to Hermicity of the matrix ${\bf W}{\bf A}$.

\begin{thm}\label{th:detrep_dmp}Let ${\bf A}\in  {\mathbb{H}}^{m\times n}_r$ and ${\bf W}\in  {\mathbb{H}}^{n\times m}$ be a nonzero matrix. Suppose  $k=\max\{\Ind ({\bf W}{\bf A}),$ $\Ind ({\bf A}{\bf W})\}$ and $\rk ({\bf W}{\bf A})^k=\rk {\bf U}^k=r_1$.
 Then the  determinantal representations of its WDMP inverse $ {\bf A}^{d,{\dag},W}= \left(a_{ij}^{d,{\dag},W}\right)$ can be expressed as follows.

(i) If ${\bf W}{\bf A}$ is an arbitrary matrix, then
 \begin{align}
a_{ij}^{d,{\dag},W}=&\label{eq:det_wdmp} {\frac{{\sum\limits_{\alpha \in I_{r,m} {\left\{ {j}
\right\}}} {{{\rm{rdet}} _{j} {\left( {({\bf A} {\bf A}^{ *}
)_{j .} (\widetilde{{\bf \omega}}  _{i  .} )}
\right)  _{\alpha} ^{\alpha} } }
}}
}{{{\sum\limits_{\beta \in I_{r,m}} {{ {\left| {
{\bf A} {\bf A}^{ *} } \right|_{\alpha
}^{\alpha} }  }}}
{{\left(
{\sum\limits_{\alpha \in I_{r_1,n}} {{\left|  {\bf U}^{ 2k+1} \left({\bf U}^{ 2k+1} \right)^{*}
 \right|_{\alpha} ^{\alpha}}}}\right)^2 }},
}}}\end{align}
where $\widetilde{{\bf \omega}}  _{i  .}$ is the $i$-th row  of $\widetilde{{\bf \Omega}}={\bf \Omega}({\bf W}{\bf A})^{k+1} {\bf A}^{ *}$. The matrix ${\bf \Omega}=(\omega_{is})$ is such that $\omega_{is}$ is determined by
\begin{align}
\omega_{is}={\sum\limits_{\alpha \in I_{r_1,n} {\left\{ {s}
\right\}}} {{\rm{rdet}} _{s} \left( {\left( { {\bf U}^{ 2k+1} \left({\bf U}^{ 2k+1} \right)^{*}
} \right)_{s.} ({\widehat {\bf \phi}}_{i.})} \right)  _{\alpha} ^{\alpha} } }
\label{eq:om1},
\end{align}
where  ${\widehat {\bf \phi}}_{i.}$ is the $i$-th row  of $
\widehat{ {\bf \Phi}}:={\bf W}{\bf A}{\bf \Phi}{\bf U}^{2k}({\bf U}^{ 2k+1})^{*}\in {\mathbb{H}}^{n\times n}$, and $
{\bf \Phi}=(\phi_{lq})\in  {\mathbb{H}}^{m\times n}$ such that
\begin{align}\label{eq:phidmp}
\phi_{iq}={\sum\limits_{\alpha \in I_{r_1,n} {\left\{ {q}
\right\}}} {{\rm{rdet}} _{q} \left( {\left( { {\bf U}^{ 2k+1} \left({\bf U}^{ 2k+1} \right)^{*}
} \right)_{q.} (\check{ {\bf u}}_{l\,.})} \right)  _{\alpha} ^{\alpha} } }.
\end{align}
Here
$\check{ {\bf u}}_{l\,.}$ is the $l$-th row of $
{\bf U}^{k}({\bf U}^{ 2k+1})^{*} =:\check{ {\bf U}}\in {\mathbb{H}}^{n\times n}$.

(ii) If ${\bf W}{\bf A}$ is Hermitian,  then
 \begin{align}
a_{ij}^{d,{\dag},W}=&\label{eq:det_wdmp2} {\frac{{\sum\limits_{\alpha \in I_{r,m} {\left\{ {j}
\right\}}} {{{\rm{rdet}} _{j} {\left( {({\bf A} {\bf A}^{ *}
)_{j .} (\widetilde{{\bf \omega}}  _{i  .} )}
\right)  _{\alpha} ^{\alpha} } }
}}
}{{{\sum\limits_{\alpha \in I_{r,m}} {{ {\left| {
{\bf A} {\bf A}^{ *} } \right|_{\alpha
}^{\alpha} }  }}}
{\sum\limits_{\alpha \in
I_{r_1,n}}  {{\left| {\left({\bf W} {{\rm {\bf A}} } \right)^{k+2}   } \right|_{\alpha} ^{\alpha}}}}
}}},\end{align}
where $\widetilde{{\bf \omega}}  _{i  .}$ is the $i$-th row  of $\widetilde{{\bf \Omega}}={\bf \Omega}{\bf W}{\bf A} {\bf A}^{ *}$. The matrix ${\bf \Omega}=(\omega_{is})$ is such that
\begin{align*}
\omega_{is}:={\sum\limits_{\alpha
\in I_{r_1,n} {\left\{ {s} \right\}}} {{\rm{rdet}} _{s} \left(
{({\bf W}{\rm {\bf A}} )^{ k+2}_{s.} ( {\bf \widehat{u}}_{i.} )}
\right)_{\alpha} ^{\alpha} } },
\end{align*}
where  $ {\bf \widehat{u}}_{i.} $  is the $i$-th row of ${\bf \widehat{U}}=\left({\bf W}{\bf A}\right)^{k+1}$.

\end{thm}
\begin{proof}
Taking into account (\ref{eq:def_dmp}), we have
\begin{equation}\label{eq:dmp_start}
a_{ij}^{d,{\dag},W} =  \sum_{t = 1}^m\sum_{s = 1}^n \sum_{f = 1}^m w_{it}{a}_{ts}^{d,W}w_{sf} p^A_{fj},
\end{equation}
where  ${\bf A}^{d,W}=(a^{d,W}_{ts})\in{\mathbb{H}}^{m \times
n}$, and ${\bf P}_A=(p^A_{fj})\in{\mathbb{H}}^{m \times
m}$.

(i)~~  Denote  ${\bf W}_1:={\bf W}{\bf A}{\bf A}^{\ast}=(w^{(1)}_{sj})$. By applying  (\ref{eq:det_repr_proj_P}) for the determinantal representation of    ${\bf P}_A$, respectively, we have

\begin{align}\nonumber
\sum_{f = 1}^m w_{sf} p^A_{fj}=&\sum_{f = 1}^m w_{sf}
{\frac{{{\sum\limits_{\alpha \in I_{r,m} {\left\{ {j}
\right\}}} {{{\rm{rdet}} _{j} {\left( {({\bf A} {\bf A}^{ *}
)_{j .} ({\bf \ddot{a}}  _{f  .} )}
\right)  _{\alpha} ^{\alpha} } }}}
}}{{{\sum\limits_{\alpha \in I_{r,m}} {{ {\left| {
{\bf A} {\bf A}^{ *} } \right| _{\alpha
}^{\alpha} }  }}} }}}\\=&{\frac{{{\sum\limits_{\alpha \in I_{r,m} {\left\{ {j}
\right\}}} {{{\rm{rdet}} _{j} {\left( {({\bf A} {\bf A}^{ *}
)_{j .} ({\bf w}^{(1)}  _{s  .} )}
\right)  _{\alpha} ^{\alpha} } }}}
}}{{{\sum\limits_{\alpha \in I_{r,m}} {{ {\left| {
{\bf A} {\bf A}^{ *} } \right| _{\alpha
}^{\alpha} }  }}} }}}.\label{eq:w1}
\end{align}
Substituting (\ref{eq:w1})  into (\ref{eq:dmp_start}) and applying (\ref{eq:det_rep_u}) for the determinantal representation of  ${\bf A}^{d,W}$ give
\begin{align}
a_{ij}^{d,{\dag},W} =\nonumber&\sum_{s = 1}^n
{\frac{\sum\limits_{\alpha \in I_{r_1,\,n} {\left\{ {s}
\right\}}} {{\rm{rdet}} _{s} \left( {\left( { {\bf U}^{ 2k+1} \left({\bf U}^{ 2k+1} \right)^{*}
} \right)_{s.} (\widehat{ {\bf \phi}}_{i.})} \right)  _{\alpha} ^{\alpha} } }{{\left(
{\sum\limits_{\alpha \in I_{r_1,\,n}} {{\left|  {\bf U}^{ 2k+1} \left({\bf U}^{ 2k+1} \right)^{*}
 \right|_{\alpha} ^{\alpha}}}}\right)^2 }}}\times\\
&{\frac{\sum\limits_{\alpha \in I_{r,m} {\left\{ {j}
\right\}}} {{{\rm{rdet}} _{j} {\left( {({\bf A} {\bf A}^{ *}
)_{j .} ({\bf w}^{(2)}  _{s  .} )}
\right)  _{\alpha} ^{\alpha} } }
}}{{{\sum\limits_{\alpha \in I_{r,m}} {{ {\left| {
{\bf A} {\bf A}^{ *} } \right| _{\alpha
}^{\alpha} }  }}} }}},
\label{eq:det_wdmp_s}
\end{align}
where  $\widehat{ {\bf \phi}}_{i\,.}$ is the $i$-th row of $
\widehat{ {\bf \Phi}}:={\bf W}{\bf A}{\bf \Phi}{\bf U}^{2k}({\bf U}^{ 2k+1})^{*}\in {\mathbb{H}}^{m\times n}$, and ${\bf w}^{(2)}  _{s  .}$ is the $s$-th row of ${\bf W}_2={\bf U}^k{\bf W}{\bf A}{\bf A}^*=\left({\bf W}{\bf A}\right)^{k+1}{\bf A}^*$.
Here
 $
{\bf \Phi}=(\phi_{lq})\in  {\mathbb{H}}^{n\times n}$ such that
$$
\phi_{lq}={\sum\limits_{\alpha \in I_{r_1,n} {\left\{ {q}
\right\}}} {{\rm{rdet}} _{q} \left( {\left( { {\bf U}^{ 2k+1} \left({\bf U}^{ 2k+1} \right)^{*}
} \right)_{q.} (\check{ {\bf u}}_{l\,.})} \right)  _{\alpha} ^{\alpha} } },
$$
where
$\check{ {\bf u}}_{l.}$ is the $l$-th row of $
{\bf U}^{k}({\bf U}^{ 2k+1})^{*} =:\check{ {\bf U}}\in {\mathbb{H}}^{n\times n}$.

Denote
$$
\omega_{is}:={\sum\limits_{\alpha \in I_{r_1,n} {\left\{ {s}
\right\}}} {{\rm{rdet}} _{s} \left( {\left( { {\bf U}^{ 2k+1} \left({\bf U}^{ 2k+1} \right)^{*}
} \right)_{s.} (\widehat{ {\bf \phi}}_{t.})} \right)  _{\alpha} ^{\alpha} } }
$$ and construct the matrix ${\bf\Omega}=(\omega_{is})$.
Since
\begin{align*}
\sum_{s = 1}^n\omega_{is}{\sum\limits_{\alpha \in I_{r,m} {\left\{ {j}
\right\}}} {{{\rm{rdet}} _{j} {\left( {({\bf A} {\bf A}^{ *}
)_{j .} ({\bf w}^{(3)}  _{s  .} )}
\right)  _{\alpha} ^{\alpha} } }
}}={\sum\limits_{\alpha \in I_{r,m} {\left\{ {j}
\right\}}} {{{\rm{rdet}} _{j} {\left( {({\bf A} {\bf A}^{ *}
)_{j .} (\widetilde{{\bf \omega}}  _{i  .} )}
\right)  _{\alpha} ^{\alpha} } }
}},
\end{align*} where $\widetilde{{\bf \omega}}  _{i  .}$ is the $i$-th row of $\widetilde{{\bf \Omega}}={\bf \Omega}{\bf W}_3={\bf \Omega}({\bf W}{\bf A})^{k+1} {\bf A}^{ *}$, then
finally, from (\ref{eq:det_wdmp_s}) it follows (\ref{eq:det_wdmp}).

(ii)~~  Applying  the determinantal representation (\ref{eq:dr_rep_wrdet}) of  ${\bf A}^{d,W}$ in (\ref{eq:dmp_start}) gives
\begin{align}
a_{ij}^{d,{\dag},W} =&
{\frac{\sum\limits_{s = 1}^n\sum\limits_{\alpha
\in I_{r_1,n} {\left\{ {s} \right\}}} {{\rm{rdet}} _{s} \left(
{({\bf W}{\rm {\bf A}} )^{ k+2}_{s.} ( {\bf \widehat{u}}_{i.} )}
\right)_{\alpha} ^{\alpha} }
{\sum\limits_{\alpha \in I_{r,m} {\left\{ {j}
\right\}}} {{{\rm{rdet}} _{j} {\left( {({\bf A} {\bf A}^{ *}
)_{j .} ({\bf w}^{(2)}  _{s  .} )}
\right)  _{\alpha} ^{\alpha} } }
}}
 }{\sum\limits_{\alpha \in
I_{r_1,n}}  {{\left| {\left({\bf W} {{\rm {\bf A}} } \right)^{k+2}   } \right|_{\alpha} ^{\alpha}}}
{\sum\limits_{\alpha \in I_{r,m}} {{ {\left| {
{\bf A} {\bf A}^{ *} } \right| _{\alpha
}^{\alpha} }  }}}} },
\label{eq:det_wdmp_sh}
\end{align}
where  $\widehat{\bf u}_{i\,.}$ is the $i$-th row of $
\widehat{ {\bf U}}:=({\bf W}{\bf A})^{ k+1}\in {\mathbb{H}}^{n\times n}$, and ${\bf w}^{(2)}  _{s  .}$ is the $s$-th row of ${\bf W}_2={\bf U}^k{\bf W}{\bf A}{\bf A}^*=\left({\bf W}{\bf A}\right)^{k+1}{\bf A}^*$.

Denote
$$
\omega_{is}:={\sum\limits_{\alpha
\in I_{r_1,n} {\left\{ {s} \right\}}} {{\rm{rdet}} _{s} \left(
{({\bf W}{\rm {\bf A}} )^{ k+2}_{s.} ( {\bf \widehat{u}}_{i.} )}
\right)_{\alpha} ^{\alpha} } }
$$ and construct the matrix ${\bf\Omega}=(\omega_{is})$.
Since
\begin{align*}
\sum_{s = 1}^n\omega_{is}{\sum\limits_{\alpha \in I_{r,m} {\left\{ {j}
\right\}}} {{{\rm{rdet}} _{j} {\left( {({\bf A} {\bf A}^{ *}
)_{j .} ({\bf w}^{(2)}  _{s  .} )}
\right)  _{\alpha} ^{\alpha} } }
}}={\sum\limits_{\alpha \in I_{r,m} {\left\{ {j}
\right\}}} {{{\rm{rdet}} _{j} {\left( {({\bf A} {\bf A}^{ *}
)_{j .} (\widetilde{{\bf \omega}}  _{i  .} )}
\right)  _{\alpha} ^{\alpha} } }
}},
\end{align*} where $\widetilde{{\bf \omega}}  _{i  .}$ is the $i$-th row of $\widetilde{{\bf \Omega}}={\bf \Omega}{\bf W}_2={\bf \Omega}({\bf W}{\bf A})^{k+1} {\bf A}^{ *}$, then,
finally, from (\ref{eq:det_wdmp_sh}) it follows (\ref{eq:det_wdmp2}).
\end{proof}
The following corollary can be get in the complex case.
\begin{cor}
Let ${\bf A}\in  {\mathbb{C}}^{m\times n}_r$ and ${\bf W}\in  {\mathbb{C}}^{n\times m}$ be a nonzero matrix. Suppose  $k=\max\{\Ind ({\bf W}{\bf A}),$ $\Ind ({\bf A}{\bf W})\}$ and $\rk ({\bf W}{\bf A})^k=\rk {\bf U}^k=r_1$.
 Then the  determinantal representations of its WDMP inverse $ {\bf A}^{d,{\dag},W}= \left(a_{ij}^{d,{\dag},W}\right)$ can be expressed as 
 \begin{align*}
a_{ij}^{d,{\dag},W}=& {\frac{{\sum\limits_{\alpha \in I_{r,m} {\left\{ {j}
\right\}}} {{ {\left| {({\bf A} {\bf A}^{ *}
)_{j .} (\widetilde{{\bf \omega}}  _{i  .} )}
\right|  _{\alpha} ^{\alpha} } }
}}
}{{{\sum\limits_{\alpha \in I_{r,m}} {{ {\left| {
{\bf A} {\bf A}^{ *} } \right|_{\alpha
}^{\alpha} }  }}}
{\sum\limits_{\alpha \in
I_{r_1,n}}  {{\left| {\left({\bf W} {{\rm {\bf A}} } \right)^{k+2}   } \right|_{\alpha} ^{\alpha}}}}
}}},\end{align*}
where $\widetilde{{\bf \omega}}  _{i  .}$ is the $i$-th row  of $\widetilde{{\bf \Omega}}={\bf \Omega}{\bf W}{\bf A} {\bf A}^{ *}$. The matrix ${\bf \Omega}=(\omega_{is})$ is such that
\begin{align*}
\omega_{is}:={\sum\limits_{\alpha
\in I_{r_1,n} {\left\{ {s} \right\}}} {\left|
{({\bf W}{\rm {\bf A}} )^{ k+2}_{s.} ( {\bf \widehat{u}}_{i.} )}
\right|_{\alpha} ^{\alpha} } },
\end{align*}
where  $ {\bf \widehat{u}}_{i.} $  is the $i$-th row of ${\bf \widehat{U}}=\left({\bf W}{\bf A}\right)^{k+1}$.
\end{cor}

For the weighted  MPD inverse, we have similarly.

\begin{thm}\label{th:detrep_mpd}Let ${\bf A}\in  {\mathbb{H}}^{m\times n}_r$ and ${\bf W}\in  {\mathbb{H}}^{n\times m}$ be a nonzero matrix. Suppose  $k=\max\{\Ind ({\bf W}{\bf A}),$ $\Ind ({\bf A}{\bf W})\}$ and
  $\rk ({\bf A}{\bf W})^k=\rk {\bf V}^k=r_1$.  Then the  determinantal representations of its WMPD inverse  $ {\bf A}^{{\dag},d,W}= \left(a_{ij}^{{\dag},d,W}\right)$ can be expressed as follows.

(i) If the matrix ${\bf A}{\bf W}$ is arbitrary, then
\begin{align}
a_{ij}^{{\dag},d,W}=&\label{eq:det_wmpd}
 {\frac{\sum\limits_{\beta \in J_{r,n} {\left\{ {i}
\right\}}} {{\rm{cdet}} _{i} \left( {\left( {{\bf A}^{ *}
{\bf A}} \right)_{.i} \left(\widetilde{{\bf \upsilon}}_{.j} \right)}
\right)  _{\beta} ^{\beta} } }{{{\sum\limits_{\beta \in  J_{r,n}} {{ {\left| {
{\bf A}^{ *} {\bf A} } \right|_{\beta
}^{\beta} }  }}}
{\sum\limits_{\beta \in J_{r_1,m}} {{\left|  {\left({\bf V}^{2k+1} \right)^{*} {\bf V}^{2k+1}
}   \right|_{\beta} ^{\beta}}}}
}}}
\end{align}
where $\widetilde{{\bf \upsilon}}_{.j}$ is the $j$-th column of $\widetilde{{\bf \Upsilon}}={\bf A}^{ *}({\bf A}{\bf W})^{k+1} {\bf \Upsilon}$. The matrix $\mathbf{\Upsilon}=(\upsilon_{tj})$ is determined by
\begin{align*}
 \upsilon_{tj}:={\sum\limits_{\beta \in J_{r_1,\,m} {\left\{ {t}
\right\}}} {{\rm{cdet}} _{t} \left( {\left(\left({\bf V}^{ 2k+1} \right)^{*}{\bf V}^{ 2k+1} \right)_{. t} \left( \widehat{ {\bf \psi}}_{.j}
\right)} \right)  _{\beta} ^{\beta} } },
\end{align*}
 where $\widehat{ {\bf \psi}}_{.j}$ is the $j$-th column of $\widehat{ {\bf \Psi}}:=({\bf V}^{ 2k+1})^{*}{\bf V}^{2k}{\bf \Psi}{\bf A}{\bf W}\in {\mathbb{H}}^{m\times n}$. Here the matrix
 $
{\bf \Psi}=(\psi_{s\,l})\in  {\mathbb{H}}^{m\times m}$ is such that
$$
\psi_{s\,l}={\sum\limits_{\beta \in J_{r_1,\,m} {\left\{ {s}
\right\}}} {{\rm{cdet}} _{s} \left( {\left(\left({\bf V}^{ 2k+1} \right)^{*}{\bf V}^{ 2k+1} \right)_{. s} \left( \hat{ {\bf v}}_{.l}
\right)} \right)  _{\beta} ^{\beta} } },
$$
where
$\hat{ {\bf v}}_{.l}$ is the $l$-th column of  $
({\bf V}^{ 2k+1})^{*}{\bf V}^{k} =:\hat{ {\bf V}}\in {\mathbb{H}}^{m\times m}$.

(ii) If the matrix ${\bf A}{\bf W}$ is Hermitian, then
\begin{align}
a_{ij}^{{\dag},d,W}=&\label{eq:det_wmpd2}
 {\frac{\sum\limits_{\beta \in J_{r,n} {\left\{ {i}
\right\}}} {{\rm{cdet}} _{i} \left( {\left( {{\bf A}^{ *}
{\bf A}} \right)_{.i} \left(\widetilde{{\bf \upsilon}}_{.j} \right)}
\right)  _{\beta} ^{\beta} } }{{{\sum\limits_{\beta \in J_{r,n}} {{ {\left| {
{\bf A}^{ *} {\bf A} } \right|_{\beta
}^{\beta} }  }}}
\left({\sum\limits_{\beta \in J_{r_1,m}} {{\left|  {\left(  {\bf A}{\bf W} \right)^{k+2}
}   \right|_{\beta} ^{\beta}}}}\right)^2
}}}
\end{align}
where $\widetilde{{\bf \upsilon}}_{.j}$ is the $j$-th column of $\widetilde{{\bf \Upsilon}}={\bf A}^{ *}({\bf A}{\bf W})^{k+1} {\bf \Upsilon}$. The matrix $\mathbf{\Upsilon}=(\upsilon_{sj})$ is determined by
\begin{align*}
\upsilon_{sj}=&{\sum\limits_{\beta
\in J_{r_1,m} {\left\{ {s} \right\}}} {{\rm{cdet}} _{s} \left(
{\left(  {\bf A}{\bf W} \right)^{k+2}_{. \,s} \left(  {
\widetilde{\bf v}}_{.j} \right)} \right) _{\beta}
^{\beta} }  },
\end{align*}
where ${\widetilde{\bf v}}_{.j} $ is the $j$-th column of  ${ \widetilde{\bf V}}=({\bf A}{\bf W})^{k+1} $.
\end{thm}
\begin{proof}Taking into account (\ref{eq:def_mpd}), we have
\begin{equation}\label{eq:mpd_start}
a_{ij}^{{\dag},d,W} =  \sum_{t = 1}^m\sum_{s = 1}^n \sum_{f = 1}^m q^A_{if} w_{ft}{a}_{ts}^{d,W}w_{sf},
\end{equation}
where  ${\bf A}^{d,W}=(a^{d,W}_{ts})\in{\mathbb{H}}^{m \times
n}$ and ${\bf Q}_A=(q^A_{if})\in{\mathbb{H}}^{n \times
n}$.

 (i)~~ Denote  ${\bf W}_1:={\bf A}^*{\bf A}{\bf W}=(w^{(1)}_{sj})$. By applying  (\ref{eq:det_repr_proj_Q}) for the determinantal representation of    ${\bf Q}_A$,  we have
\begin{align}\nonumber
\sum_{l = 1}^n q^A_{il}w_{lt}=&\sum_{l = 1}^n {\frac{{{\sum\limits_{\beta \in J_{r,n} {\left\{ {i}
\right\}}} {{\rm{cdet}} _{i} \left( {\left( {{\bf A}^{ *}
{\bf A}} \right)_{.i} \left({\bf \dot{a}}_{.l} \right)}
\right)  _{\beta} ^{\beta} } }}}{{{\sum\limits_{\beta
\in J_{r,n}}  {{ \left| {{\bf A}^{ *}  {\bf
A}} \right|_{\beta}^{\beta} }}}} }}w_{lt}=\\=& {\frac{{{\sum\limits_{\beta \in J_{r,n} {\left\{ {i}
\right\}}} {{\rm{cdet}} _{i} \left( {\left( {{\bf A}^{ *}
{\bf A}} \right)_{.i} \left({\bf w}^{(1)}_{.t} \right)}
\right)  _{\beta} ^{\beta} } }}}{{{\sum\limits_{\beta
\in J_{r,n}}  {{ \left| {{\bf A}^{ *}  {\bf
A}} \right|_{\beta}^{\beta} }}}} }}\label{eq:w1mdp}
\end{align}
Substituting (\ref{eq:w1mdp})  into (\ref{eq:mpd_start}) and applying (\ref{eq:det_rep_v}) for the determinantal representation of  ${\bf A}^{d,W}$,  we  obtain

\begin{align}\nonumber
a_{ij}^{{\dag},d,W} =&\sum_{t = 1}^m{\frac{\sum\limits_{\beta \in J_{r,n} {\left\{ {i}
\right\}}} {{\rm{cdet}} _{i} \left( {\left( {{\bf A}^{ *}
{\bf A}} \right)_{.i} \left({\bf w}^{(2)}_{.t} \right)}
\right)  _{\beta} ^{\beta} } }{{{\sum\limits_{\beta
\in J_{r,n}}  {{ \left| {{\bf A}^{ *}  {\bf
A}} \right|_{\beta}^{\beta} }}}} }}\times \\&
{\frac{{    {\sum\limits_{\beta \in J_{r_1,\,m} {\left\{ {t}
\right\}}} {{\rm{cdet}} _{t} \left( {\left(\left({\bf V}^{ 2k+1} \right)^{*}{\bf V}^{ 2k+1} \right)_{. t} \left( \widehat{ {\bf \psi}}_{.j}
\right)} \right)  _{\beta} ^{\beta} } }
}}{{\left({\sum\limits_{\beta \in J_{r_1,\,m}} {{\left| \left({\bf V}^{ 2k+1} \right)^{*}{\bf V}^{ 2k+1}
  \right|_{\beta} ^{\beta}}}} \right)^2}}},
\label{eq:det_wmpd_s}
\end{align}
where $\widehat{ {\bf \psi}}_{.j}$ is the $j$-th column of $\widehat{ {\bf \Psi}}:=({\bf V}^{ 2k+1})^{*}{\bf V}^{2k}{\bf \Psi}{\bf A}{\bf W}\in {\mathbb{H}}^{m\times n}$
and ${\bf w}^{(2)}_{.t} $ is the $t$-th column of  ${\bf W}_2={\bf A}^*{\bf A}{\bf W}{\bf V}^k={\bf A}^*\left({\bf A}{\bf W}\right)^{k+1}\in {\mathbb{H}}^{m\times m}$. Here the matrix
 $
{\bf \Psi}=(\psi_{s\,l})\in  {\mathbb{H}}^{m\times m}$ is such that

$$
\psi_{s\,l}={\sum\limits_{\beta \in J_{r_1,\,m} {\left\{ {s}
\right\}}} {{\rm{cdet}} _{s} \left( {\left(\left({\bf V}^{ 2k+1} \right)^{*}{\bf V}^{ 2k+1} \right)_{. s} \left( \hat{ {\bf v}}_{.l}
\right)} \right)  _{\beta} ^{\beta} } },
$$
here
$\hat{ {\bf v}}_{.l}$ is the $l$-th column of  $
({\bf V}^{ 2k+1})^{*}{\bf V}^{k} =:\hat{ {\bf V}}\in {\mathbb{H}}^{m\times m}$.

Denote $$
 \upsilon_{tj}:={\sum\limits_{\beta \in J_{r_1,\,m} {\left\{ {t}
\right\}}} {{\rm{cdet}} _{t} \left( {\left(\left({\bf V}^{ 2k+1} \right)^{*}{\bf V}^{ 2k+1} \right)_{. t} \left( \widehat{ {\bf \psi}}_{.j}
\right)} \right)  _{\beta} ^{\beta} } }
$$ and construct the matrix ${\bf {  \Upsilon}}=( \upsilon_{tj})$.
Taking into account that
\begin{align*}
\sum_{t = 1}^m{\sum\limits_{\beta \in J_{r,n} {\left\{ {i}
\right\}}} {{\rm{cdet}} _{i} \left( {\left( {{\bf A}^{ *}
{\bf A}} \right)_{.i} \left({\bf w}^{(2)}_{.t} \right)}
\right)  _{\beta} ^{\beta} } }\upsilon_{tj}={\sum\limits_{\beta \in J_{r,n} {\left\{ {i}
\right\}}} {{\rm{cdet}} _{i} \left( {\left( {{\bf A}^{ *}
{\bf A}} \right)_{.i} \left(\widetilde{{\bf \upsilon}}_{.j} \right)}
\right)  _{\beta} ^{\beta} } },
\end{align*} where $\widetilde{{\bf \upsilon}}_{.j}$ is the $j$-th column of $\widetilde{{\bf \Upsilon}}={\bf W}_2{\bf \Upsilon}={\bf A}^{ *}({\bf A}{\bf W})^{k+1} {\bf \Upsilon}$,
finally, from (\ref{eq:det_wmpd_s}) it follows (\ref{eq:det_wmpd}).

(ii)~~ Applying the determinantal representation (\ref{eq:dr_rep_wcdet}) of ${\bf A}^{d,W}$ in  (\ref{eq:mpd_start}) gives
\begin{align}\nonumber
a_{ij}^{{\dag},d,W} =&\sum_{t = 1}^m{\frac{\sum\limits_{\beta \in J_{r,n} {\left\{ {i}
\right\}}} {{\rm{cdet}} _{i} \left( {\left( {{\bf A}^{ *}
{\bf A}} \right)_{.i} \left({\bf w}^{(2)}_{.t} \right)}
\right)  _{\beta} ^{\beta} } }{{{\sum\limits_{\beta
\in J_{r,n}}  {{ \left| {{\bf A}^{ *}  {\bf
A}} \right|_{\beta}^{\beta} }}}} }}\times \\&
{\frac{\sum\limits_{\beta
\in J_{r_1,m} {\left\{ {t} \right\}}} {{\rm{cdet}} _{t} \left(
{\left(  {\bf A}{\bf W} \right)^{k+2}_{. \,t} \left(  {
\widetilde{\bf v}}_{.j} \right)} \right) _{\beta}
^{\beta} }  }{\sum\limits_{\beta \in J_{r_1,m}} {{\left|
{\left( { {\bf A}{\bf W}} \right)^{k+2}}  \right|_{\beta} ^{\beta}}} }},
\label{eq:det_wmpdh_s}
\end{align}
where ${\bf w}^{(2)}_{.t} $ is the $t$-th column of  ${\bf W}_2={\bf A}^*{\bf A}{\bf W}{\bf V}^k={\bf A}^*\left({\bf A}{\bf W}\right)^{k+1}\in {\mathbb{H}}^{m\times m}$ and   ${\widetilde{\bf v}}_{.j} $ is the $j$-th column of  ${ \widetilde{\bf V}}=({\bf A}{\bf W})^{k+1} $.
Now, construct the matrix ${\bf {  \Upsilon}}=( \upsilon_{tj})$, where
$$ \upsilon_{tj}:={\sum\limits_{\beta
\in J_{r_1,m} {\left\{ {t} \right\}}} {{\rm{cdet}} _{t} \left(
{\left(  {\bf A}{\bf W} \right)^{k+2}_{. \,t} \left(  {
\widetilde{\bf v}}_{.j} \right)} \right) _{\beta}
^{\beta} }  }.$$ Then, from (\ref{eq:det_wmpdh_s}) it follows (\ref{eq:det_wmpd2}).

\end{proof}

\begin{cor}Let ${\bf A}\in  {\mathbb{C}}^{m\times n}_r$ and ${\bf W}\in  {\mathbb{C}}^{n\times m}$ be a nonzero matrix. Suppose  $k=\max\{\Ind ({\bf W}{\bf A}),$ $\Ind ({\bf A}{\bf W})\}$ and
  $\rk ({\bf A}{\bf W})^k=\rk {\bf V}^k=r_1$.  Then the  determinantal representations of its WMPD inverse  $ {\bf A}^{{\dag},d,W}= \left(a_{ij}^{{\dag},d,W}\right)$ can be expressed as 
\begin{align*}
a_{ij}^{{\dag},d,W}=&
 {\frac{\sum\limits_{\beta \in J_{r,n} {\left\{ {i}
\right\}}} { \left| {\left( {{\bf A}^{ *}
{\bf A}} \right)_{.i} \left(\widetilde{{\bf \upsilon}}_{.j} \right)}
\right|  _{\beta} ^{\beta} } }{{{\sum\limits_{\beta \in J_{r,n}} {{ {\left| {
{\bf A}^{ *} {\bf A} } \right|_{\beta
}^{\beta} }  }}}
\left({\sum\limits_{\beta \in J_{r_1,m}} {{\left|  {\left(  {\bf A}{\bf W} \right)^{k+2}
}   \right|_{\beta} ^{\beta}}}}\right)^2
}}}
\end{align*}
where $\widetilde{{\bf \upsilon}}_{.j}$ is the $j$-th column of $\widetilde{{\bf \Upsilon}}={\bf A}^{ *}({\bf A}{\bf W})^{k+1} {\bf \Upsilon}$. The matrix $\mathbf{\Upsilon}=(\upsilon_{sj})$ is determined by
\begin{align*}
\upsilon_{sj}=&{\sum\limits_{\beta
\in J_{r_1,m} {\left\{ {s} \right\}}} { \left|
{\left(  {\bf A}{\bf W} \right)^{k+2}_{. \,s} \left(  {
\widetilde{\bf v}}_{.j} \right)} \right| _{\beta}
^{\beta} }  },
\end{align*}
where ${\widetilde{\bf v}}_{.j} $ is the $j$-th column of  ${ \widetilde{\bf V}}=({\bf A}{\bf W})^{k+1} $.
\end{cor}

Theorems  \ref{th:detrep_dmp} and \ref{th:detrep_mpd} give the  determinantal representations of the weighted DMP and DMP inverses over the quaternion skew field.
For better understanding, we present  the algorithm of finding one of them, for example WDMP from Theorem \ref{th:detrep_dmp} in the case (i).

\begin{alg}\label{al1}
\begin{enumerate}

    \item Compute the matrix $\check{ {\bf U}}=
{\bf U}^{k}({\bf U}^{ 2k+1})^{*}.$
    \item Find     $\phi_{iq}$ by (\ref{eq:phidmp}) for all $i,q=1,\ldots,n$ and construct the matrix $
{\bf \Phi}=(\phi_{iq})$.
        \item Compute the matrix $
\widehat{ {\bf \Phi}}:={\bf W}{\bf A}{\bf \Phi}{\bf U}^{2k}({\bf U}^{ 2k+1})^{*}$.

 \item By   (\ref{eq:om1}), find  $\omega_{is}$   for all $i,s=1,\ldots,n$ and construct the matrix ${\bf \Omega}=(\omega_{is})$.
\item Compute the matrix $
\widetilde{ {\bf \Omega}}:={\bf \Omega}\left({\bf W}{\bf \Omega}\right)^{k+1}{\bf A}^{*}$.
         \item Finally, find $a_{ij}^{d,{\dag},W}$  by  (\ref{eq:det_wdmp}) for all $i=1,\ldots,m$ and $j=1,\ldots,n$ .

  \end{enumerate}
\end{alg}

\section{Determinantal representations of the weighted CMP inverse}
In \cite{meh} by M. Mehdipour and A. Salemi the  CMP inverse was investigated that can be extended to quaternion matrices as follows.
\begin{defn}\cite{ky_cor} Suppose ${\bf A}\in  {\mathbb{H}}^{n\times n}$ has the core-nilpotent decomposition   ${\bf A}= {\bf A}_1 +{\bf A}_2$, where $\Ind{\bf A}_1=\Ind{\bf A}$, ${\bf A}_2$ is nilpotent and
${\bf A}_1{\bf A}_2={\bf A}_2{\bf A}_1=0$.
The CMP inverse of ${\bf A}$ is called the matrix  ${\bf A}^{c,\dag}:={\bf A}^{\dag}{\bf A}_1{\bf A}^{\dag}$.
\end{defn}
Similarly to complex matrices  can be proved the next lemma.
\begin{lem}
Let ${\bf A}\in  {\mathbb{H}}^{n\times n}$. The matrix ${\bf X}={\bf A}^{c,\dag}$ is the unique matrix that
satisfies the following system of equations:
\begin{align*}
{\bf X}{\bf A}{\bf X}={\bf X},~{\bf A}{\bf X}{\bf A}={\bf A}_1,~{\bf A}{\bf X}={\bf A}_1{\bf A}^{\dag},~and~ {\bf X}{\bf A}={\bf A}^{\dag}{\bf A}_1. \end{align*}
Moreover,  $
{\bf A}^{c,\dag}={\bf A}^{\dag}{\bf A}{\bf A}^d{\bf A}{\bf A}^{\dag}.$
\end{lem}
Determinantal representations of the CMP inverse over the quaternion skew field within the framework of the theory of row-column determinants are derived in \cite{ky_cor}.

Recently, Mosi\`{c} \cite{mos2} introduced the weighted CMP
inverse of a rectangular matrix that can be  extended over the quaternion skew field
without any changes.
\begin{lem}
Let ${\bf A}\in  {\mathbb{H}}^{m\times n}$ and ${\bf W}\in  {\mathbb{H}}^{n\times m}$ be a nonzero matrix. The
system of equations
\begin{align*}
{\bf X}{\bf A}{\bf X}={\bf X},~{\bf A}{\bf X}={\bf A}{\bf W}{\bf A}^{d,W}{\bf W}{\bf A}{\bf A}^{\dag},~and~ {\bf X}{\bf A}={\bf A}^{\dag}{\bf A}{\bf W}{\bf A}^{d,W}{\bf W}{\bf A}. \end{align*}
is consistent and its unique solution is  $
{\bf X}={\bf A}^{\dag}{\bf A}{\bf W}{\bf A}^{d,W}{\bf W}{\bf A}{\bf A}^{\dag}.$
\end{lem}
\begin{defn}\label{def:wcmp}
Let ${\bf A}\in  {\mathbb{H}}^{m\times n}$ and ${\bf W}\in  {\mathbb{H}}^{n\times m}$ be a nonzero matrix.  The weighted CMP (WCMP) inverse of ${\bf A}$ with respect to  ${\bf W}$ is defined as
\begin{align}\label{eq:wcmp}
{\bf A}^{c,\dag,W}={\bf A}^{\dag}{\bf A}{\bf W}{\bf A}^{d,W}{\bf W}{\bf A}{\bf A}^{\dag}. \end{align}
\end{defn}

Taking into account Corollary \ref{cor:det_repr_proj_Q} and Lemma  \ref{theor:det_rep_wdraz1}, it follows the next theorem  about determinantal representations of the quaternion WCMP inverse.

\begin{thm}\label{th:detrep_cmp}Let ${\bf A}\in  {\mathbb{H}}^{m\times n}_r$ and ${\bf W}\in  {\mathbb{H}}^{n\times m}$ be a nonzero matrix. Suppose  $k=\max\{\Ind ({\bf W}{\bf A}),$ $\Ind ({\bf A}{\bf W})\}$.
 Then the  determinantal representations of its WCMP inverse $ {\bf A}^{c,{\dag},W}= \left(a_{ij}^{c,{\dag},W}\right)$ can be expressed as

(i)~ if  $\rk ({\bf W}{\bf A})^k=\rk {\bf U}^k=r_1$, then
 \begin{align}
a_{ij}^{c,{\dag},W}=&\label{eq:det_wcmp1} {\frac{{\sum\limits_{\alpha \in I_{r,m} {\left\{ {j}
\right\}}} {{{\rm{rdet}} _{j} {\left( {({\bf A} {\bf A}^{ *}
)_{j .} (\widetilde{{\bf \omega}}  _{i  .} )}
\right)  _{\alpha} ^{\alpha} } }
}}
}{{\left({\sum\limits_{\beta \in J_{r,n}} {{ {\left| {
{\bf A}^{ *} {\bf A} } \right|_{\beta
}^{\beta} }  }}}\right)^2
{{\left(
{\sum\limits_{\alpha \in I_{r_1,n}} {{\left|  {\bf U}^{ 2k+1} \left({\bf U}^{ 2k+1} \right)^{*}
 \right|_{\alpha} ^{\alpha}}}}\right)^2 }},
}}}\end{align}
where $\widetilde{{\bf \omega}}  _{i  .}$ is the $i$-th row  of $\widetilde{{\bf \Omega}}={\bf \Omega}({\bf W}{\bf A})^{k+1} {\bf A}^{ *}$. The matrix ${\bf \Omega}=(\omega_{iz})$ is such that $\omega_{iz}$ is determined by
\begin{align}
\omega_{iz}={\sum\limits_{\beta \in J_{r,n} {\left\{ {i}
\right\}}} {{\rm{cdet}} _{i} \left( {\left( {{\bf A}^{ *}
{\bf A}} \right)_{.i} \left({\bf \phi}^{(1)}_{.z} \right)}
\right)  _{\beta} ^{\beta} } }\label{eq:om2}.
\end{align}
Here  ${\bf \phi}^{(1)}_{.z}$ is the $z$-th column of ${\bf \Phi}_1={\bf A}^*{\bf A}{\bf W}\widehat{\bf \Phi}=(\widehat{\phi}_{tz})$ and
\begin{align}
\widehat{\phi}_{tz}:={\sum\limits_{\alpha \in I_{r_1,n} {\left\{ {z}
\right\}}} {{\rm{rdet}} _{z} \left( {\left( { {\bf U}^{ 2k+1} \left({\bf U}^{ 2k+1} \right)^{*}
} \right)_{z.} (\widetilde{ {\bf \phi}}_{t.})} \right)  _{\alpha} ^{\alpha} } }
\label{eq:phih},
\end{align}
where  $\widetilde{ {\bf \phi}}_{t.}$ is the $t$-th row of $
\widetilde{ {\bf \Phi}}:={\bf A}{\bf \Phi}{\bf U}^{2k}({\bf U}^{ 2k+1})^{*}\in {\mathbb{H}}^{m\times n}$, and $
{\bf \Phi}=(\phi_{lq})\in  {\mathbb{H}}^{n\times n}$ such that
\begin{align}\label{eq:phi1}
\phi_{lq}={\sum\limits_{\alpha \in I_{r_1,n} {\left\{ {q}
\right\}}} {{\rm{rdet}} _{q} \left( {\left( { {\bf U}^{ 2k+1} \left({\bf U}^{ 2k+1} \right)^{*}
} \right)_{q.} (\check{ {\bf u}}_{l.})} \right)  _{\alpha} ^{\alpha} } }.
\end{align}
Here
$\check{ {\bf u}}_{l.}$ is the $l$-th row of $
{\bf U}^{k}({\bf U}^{ 2k+1})^{*} =:\check{ {\bf U}}\in {\mathbb{H}}^{n\times n}$.

(ii)~ if  $\rk ({\bf A}{\bf W})^k=\rk {\bf V}^k=r_1$, then
\begin{align}
a_{ij}^{c,{\dag},W}=&\label{eq:det_wcmp2}
 {\frac{\sum\limits_{\beta \in J_{r,n} {\left\{ {i}
\right\}}} {{\rm{cdet}} _{i} \left( {\left( {{\bf A}^{ *}
{\bf A}} \right)_{.i} \left(\widetilde{{\bf \upsilon}}_{.j} \right)}
\right)  _{\beta} ^{\beta} } }{{\left({\sum\limits_{\alpha \in I_{r,m}} {{ {\left| {
{\bf A} {\bf A}^{ *} } \right|_{\alpha
}^{\alpha} }  }}}\right)^2
{\sum\limits_{\beta \in J_{r_1,m}} {{\left|  {\left({\bf V}^{2k+1} \right)^{*} {\bf V}^{2k+1}
}   \right|_{\beta} ^{\beta}}}}
}}}
\end{align}
where $\widetilde{{\bf \upsilon}}_{.j}$ is the $j$-th column of $\widetilde{{\bf \Upsilon}}={\bf A}^{ *}({\bf A}{\bf W})^{k+1} {\bf \Upsilon}$. The matrix $\mathbf{\Upsilon}=(\upsilon_{zj})$ is determined by

\begin{align*}
\upsilon_{zj}=&{\sum\limits_{\alpha \in I_{r,m} {\left\{ {j}
\right\}}} {{{\rm{rdet}} _{j} {\left( {({\bf A} {\bf A}^{ *}
)_{j .} ({\bf { \psi}}^{(1)}  _{z  .} )}
\right)  _{\alpha} ^{\alpha} } }
}},
\end{align*}
where ${\bf { \psi}}^{(1)}  _{z  .}$ is the $z$-th row of ${\bf { \Psi}}^{(1)}=\widehat{\bf { \Psi}}{\bf W}{\bf A}{\bf A}^*$. Here $\widehat{\bf { \Psi}}=( \widehat{ { \psi}}_{zs})$ is such that
 $$
 \widehat{ { \psi}}_{zs}:={\sum\limits_{\beta \in J_{r_1,\,m} {\left\{ {z}
\right\}}} {{\rm{cdet}} _{z} \left( {\left(\left({\bf V}^{ 2k+1} \right)^{*}{\bf V}^{ 2k+1} \right)_{. z} \left( \widetilde{ {\bf \psi}}_{.s}
\right)} \right)  _{\beta} ^{\beta} } }
$$
where $ \widetilde{ {\bf \psi}}_{.s}$ is the $s$-th column of  $
\widetilde{ {\bf \Psi}}:=({\bf V}^{ 2k+1})^{*}{\bf V}^{2k}{\bf \Psi}{\bf A}\in {\mathbb{H}}^{m\times n}$, and $
{\bf \Psi}=(\psi_{lt})\in  {\mathbb{H}}^{m\times m}$ is determined by

$$
\psi_{lt}={\sum\limits_{\beta \in J_{r_1,m} {\left\{ {l}
\right\}}} {{\rm{cdet}} _{l} \left( {\left(\left({\bf V}^{ 2k+1} \right)^{*}{\bf V}^{ 2k+1} \right)_{. l} \left( \hat{ {\bf v}}_{.t}
\right)} \right)  _{\beta} ^{\beta} } }.
$$
Here
$\hat{ {\bf v}}_{.t}$ is the $t$-th column of  $
({\bf V}^{ 2k+1})^{*}{\bf V}^{k} =:\hat{ {\bf V}}\in {\mathbb{H}}^{m\times m}$.
\end{thm}
\begin{proof}
Taking into account (\ref{eq:wcmp}), we have
\begin{equation}\label{eq:cmp_start}
a_{ij}^{c,{\dag},W} =  \sum_{l = 1}^n \sum_{t = 1}^m\sum_{s = 1}^n \sum_{f = 1}^m q^A_{il}w_{lt}{a}_{ts}^{d,W}w_{sf} p^A_{fj},
\end{equation}
where ${\bf Q}_A=(q^A_{il})\in{\mathbb{H}}^{n \times
n}$, ${\bf A}^{d,W}=(a^{d,W}_{ts})\in{\mathbb{H}}^{m \times
n}$, and ${\bf P}_A=(p^A_{fj})\in{\mathbb{H}}^{m \times
m}$.

(i)  Denote ${\bf W}_1:={\bf A}^*{\bf A}{\bf W}=(w^{(1)}_{it})$ and ${\bf W}_2:={\bf W}{\bf A}{\bf A}^*=(w^{(2)}_{sj})$. By applying  one of the cases of (\ref{eq:det_repr_proj_Q}) and (\ref{eq:det_repr_proj_P}) for the determinantal representations of   ${\bf Q}_A$ and ${\bf P}_A$, respectively, we have
\begin{align}\nonumber
\sum_{l = 1}^n q^A_{il}w_{lt}=&\sum_{l = 1}^n {\frac{{{\sum\limits_{\beta \in J_{r,n} {\left\{ {i}
\right\}}} {{\rm{cdet}} _{i} \left( {\left( {{\bf A}^{ *}
{\bf A}} \right)_{.i} \left({\bf \dot{a}}_{.l} \right)}
\right)  _{\beta} ^{\beta} } }}}{{{\sum\limits_{\beta
\in J_{r,n}}  {{ \left| {{\bf A}^{ *}  {\bf
A}} \right|_{\beta}^{\beta} }}}} }}w_{lt}=\\=& {\frac{{{\sum\limits_{\beta \in J_{r,n} {\left\{ {i}
\right\}}} {{\rm{cdet}} _{i} \left( {\left( {{\bf A}^{ *}
{\bf A}} \right)_{.i} \left({\bf w}^{(1)}_{.t} \right)}
\right)  _{\beta} ^{\beta} } }}}{{{\sum\limits_{\beta
\in J_{r,n}}  {{ \left| {{\bf A}^{ *}  {\bf
A}} \right|_{\beta}^{\beta} }}}} }},\label{eq:ad1}
\\ \nonumber
\sum_{f = 1}^m w_{sf} p^A_{fj}=&\sum_{f = 1}^m w_{sf}
{\frac{{{\sum\limits_{\alpha \in I_{r,m} {\left\{ {j}
\right\}}} {{{\rm{rdet}} _{j} {\left( {({\bf A} {\bf A}^{ *}
)_{j .} ({\bf \ddot{a}}  _{f  .} )}
\right)  _{\alpha} ^{\alpha} } }}}
}}{{{\sum\limits_{\alpha \in I_{r,m}} {{ {\left| {
{\bf A} {\bf A}^{ *} } \right| _{\alpha
}^{\alpha} }  }}} }}}\\=&{\frac{{{\sum\limits_{\alpha \in I_{r,m} {\left\{ {j}
\right\}}} {{{\rm{rdet}} _{j} {\left( {({\bf A} {\bf A}^{ *}
)_{j .} ({\bf w}^{(2)}  _{s  .} )}
\right)  _{\alpha} ^{\alpha} } }}}
}}{{{\sum\limits_{\alpha \in I_{r,m}} {{ {\left| {
{\bf A} {\bf A}^{ *} } \right| _{\alpha
}^{\alpha} }  }}} }}}.\label{eq:ad2}
\end{align}
Substituting (\ref{eq:ad1}) and (\ref{eq:ad2}) into (\ref{eq:cmp_start}), denoting ${\bf W}_3=\left({\bf W}{\bf A}\right)^{k+1}{\bf A}^*$,  and applying (\ref{eq:det_rep_u}) for the determinantal representation of  ${\bf A}^{d,W}$ give

\begin{align}\nonumber
a_{ij}^{c,{\dag},W} =&\sum_{t = 1}^m\sum_{s = 1}^n{\frac{{{\sum\limits_{\beta \in J_{r,n} {\left\{ {i}
\right\}}} {{\rm{cdet}} _{i} \left( {\left( {{\bf A}^{ *}
{\bf A}} \right)_{.i} \left({\bf w}^{(1)}_{.t} \right)}
\right)  _{\beta} ^{\beta} } }}}{{{\sum\limits_{\beta
\in J_{r,n}}  {{ \left| {{\bf A}^{ *}  {\bf
A}} \right|_{\beta}^{\beta} }}}} }}\times \\&\nonumber
{\frac{\sum\limits_{z = 1}^{n}\left({{\sum\limits_{\alpha \in I_{r_1,n} {\left\{ {z}
\right\}}} {{\rm{rdet}} _{z} \left( {\left( { {\bf U}^{ 2k+1} \left({\bf U}^{ 2k+1} \right)^{*}
} \right)_{z.} (\widetilde{ {\bf \phi}}_{t.})} \right)  _{\alpha} ^{\alpha} } }
}\right){u}_{zs}^{(k)}}{{\left(
{\sum\limits_{\alpha \in I_{r_1,n}} {{\left|  {\bf U}^{ 2k+1} \left({\bf U}^{ 2k+1} \right)^{*}
 \right|_{\alpha} ^{\alpha}}}}\right)^2 }}}\times\\
&{\frac{{{\sum\limits_{\alpha \in I_{r,m} {\left\{ {j}
\right\}}} {{{\rm{rdet}} _{j} {\left( {({\bf A} {\bf A}^{ *}
)_{j .} ({\bf w}^{(2)}  _{s  .} )}
\right)  _{\alpha} ^{\alpha} } }}}
}}{{{\sum\limits_{\alpha \in I_{r,m}} {{ {\left| {
{\bf A} {\bf A}^{ *} } \right| _{\alpha
}^{\alpha} }  }}} }}}=\nonumber \\=\nonumber &\sum_{t = 1}^m\sum_{z = 1}^n{\frac{\sum\limits_{\beta \in J_{r,n} {\left\{ {i}
\right\}}} {{\rm{cdet}} _{i} \left( {\left( {{\bf A}^{ *}
{\bf A}} \right)_{.i} \left({\bf w}^{(1)}_{.t} \right)}
\right)  _{\beta} ^{\beta} } }{{{\sum\limits_{\beta
\in J_{r,n}}  {{ \left| {{\bf A}^{ *}  {\bf
A}} \right|_{\beta}^{\beta} }}}} }}\times\end{align}\begin{align} \nonumber
 &{\frac{\sum\limits_{\alpha \in I_{r_1,n} {\left\{ {z}
\right\}}} {{\rm{rdet}} _{z} \left( {\left( { {\bf U}^{ 2k+1} \left({\bf U}^{ 2k+1} \right)^{*}
} \right)_{z.} (\widetilde{ {\bf \phi}}_{t.})} \right)  _{\alpha} ^{\alpha} } }{{\left(
{\sum\limits_{\alpha \in I_{r_1,n}} {{\left|  {\bf U}^{ 2k+1} \left({\bf U}^{ 2k+1} \right)^{*}
 \right|_{\alpha} ^{\alpha}}}}\right)^2 }}}\times\\
&{\frac{\sum\limits_{\alpha \in I_{r,m} {\left\{ {j}
\right\}}} {{{\rm{rdet}} _{j} {\left( {({\bf A} {\bf A}^{ *}
)_{j .} ({\bf w}^{(3)}  _{z  .} )}
\right)  _{\alpha} ^{\alpha} } }
}}{{{\sum\limits_{\alpha \in I_{r,m}} {{ {\left| {
{\bf A} {\bf A}^{ *} } \right| _{\alpha
}^{\alpha} }  }}} }}},
\label{eq:det_wcmp_s}
\end{align}
where  $\widetilde{ {\bf \phi}}_{t.}$ is the $t$-th row of $
\widetilde{ {\bf \Phi}}:={\bf A}{\bf \Phi}{\bf U}^{2k}({\bf U}^{ 2k+1})^{*}\in {\mathbb{H}}^{m\times n}$, and $
{\bf \Phi}=(\phi_{lq})\in  {\mathbb{H}}^{m\times n}$ is such that
\begin{align*}
\phi_{lq}={\sum\limits_{\alpha \in I_{r_1,n} {\left\{ {q}
\right\}}} {{\rm{rdet}} _{q} \left( {\left( { {\bf U}^{ 2k+1} \left({\bf U}^{ 2k+1} \right)^{*}
} \right)_{q.} (\check{ {\bf u}}_{l.})} \right)  _{\alpha} ^{\alpha} } }.
\end{align*}
Here
$\check{ {\bf u}}_{l.}$ is the $l$-th row of $
{\bf U}^{k}({\bf U}^{ 2k+1})^{*} =:\check{ {\bf U}}\in {\mathbb{H}}^{m\times n}$.
Denote
$$
\widehat{\phi}_{tz}:={\sum\limits_{\alpha \in I_{r_1,n} {\left\{ {z}
\right\}}} {{\rm{rdet}} _{z} \left( {\left( { {\bf U}^{ 2k+1} \left({\bf U}^{ 2k+1} \right)^{*}
} \right)_{z.} (\widetilde{ {\bf \phi}}_{t.})} \right)  _{\alpha} ^{\alpha} } }
$$ and construct the matrix $\widehat{\bf\Phi}=(\widehat{\phi}_{tz})$. Then, determine

\begin{align*}
\omega_{iz}=&\sum_{t = 1}^n{\sum\limits_{\beta \in J_{r,n} {\left\{ {i}
\right\}}} {{\rm{cdet}} _{i} \left( {\left( {{\bf A}^{ *}
{\bf A}} \right)_{.i} \left({\bf w}^{(1)}_{.t} \right)}
\right)  _{\beta} ^{\beta} } }\widehat{\phi}_{tz}=\end{align*}\begin{align*}=&{\sum\limits_{\beta \in J_{r,n} {\left\{ {i}
\right\}}} {{\rm{cdet}} _{i} \left( {\left( {{\bf A}^{ *}
{\bf A}} \right)_{.i} \left({\bf \phi}^{(1)}_{.z} \right)}
\right)  _{\beta} ^{\beta} } }
\end{align*}
where ${\bf \phi}^{(1)}_{.t}$ is the $t$-th column of ${\bf \Phi}_1={\bf W}_1 \widehat{\bf\Phi}={\bf A}^*{\bf A}{\bf W}\widehat{\bf\Phi}$ and construct the matrix ${\bf\Omega}=({\omega}_{iz})$.
Taking into account  ${\sum\limits_{\alpha \in I_{r,m}} {{ {\left| {
{\bf A} {\bf A}^{ *} } \right| _{\alpha
}^{\alpha} }  }}}={\sum\limits_{\beta
\in J_{r,n}}  {{ \left| {{\bf A}^{ *}  {\bf
A}} \right|_{\beta}^{\beta} }}}$,
and
\begin{align*}
\sum_{z = 1}^n\omega_{iz}{\sum\limits_{\alpha \in I_{r,m} {\left\{ {j}
\right\}}} {{{\rm{rdet}} _{j} {\left( {({\bf A} {\bf A}^{ *}
)_{j .} ({\bf w}^{(3)}  _{z  .} )}
\right)  _{\alpha} ^{\alpha} } }
}}={\sum\limits_{\alpha \in I_{r,m} {\left\{ {j}
\right\}}} {{{\rm{rdet}} _{j} {\left( {({\bf A} {\bf A}^{ *}
)_{j .} (\widetilde{{\bf \omega}}  _{i  .} )}
\right)  _{\alpha} ^{\alpha} } }
}},
\end{align*} where $\widetilde{{\bf \omega}}  _{i  .}$ is the $i$-th row of $\widetilde{{\bf \Omega}}={\bf \Omega}{\bf W}_3={\bf \Omega}({\bf W}{\bf A})^{k+1} {\bf A}^{ *}$,
finally from (\ref{eq:det_wcmp_s}), it follows (\ref{eq:det_wcmp1}).

(ii) By applying  the determinantal representations (\ref{eq:det_rep_v}) for  ${\bf A}^{d,W}$ and the same as in the above point for   ${\bf Q}_A$ and ${\bf P}_A$ , we get

\begin{align}\nonumber
a_{ij}^{c,{\dag},W} =&\sum_{t = 1}^m\sum_{s = 1}^n{\frac{{{\sum\limits_{\beta \in J_{r,n} {\left\{ {i}
\right\}}} {{\rm{cdet}} _{i} \left( {\left( {{\bf A}^{ *}
{\bf A}} \right)_{.i} \left({\bf w}^{(1)}_{.t} \right)}
\right)  _{\beta} ^{\beta} } }}}{{{\sum\limits_{\beta
\in J_{r,n}}  {{ \left| {{\bf A}^{ *}  {\bf
A}} \right|_{\beta}^{\beta} }}}} }}\times\end{align} \begin{align}&\nonumber
{\frac{{ \sum\limits_{z = 1}^{m} {v}_{tz}^{(k)}   {\sum\limits_{\beta \in J_{r_1,\,m} {\left\{ {z}
\right\}}} {{\rm{cdet}} _{z} \left( {\left(\left({\bf V}^{ 2k+1} \right)^{*}{\bf V}^{ 2k+1} \right)_{. z} \left( \widetilde{ {\bf \psi}}_{.s}
\right)} \right)  _{\beta} ^{\beta} } }
}}{{\left({\sum\limits_{\beta \in J_{r_1,\,m}} {{\left| \left({\bf V}^{ 2k+1} \right)^{*}{\bf V}^{ 2k+1}
  \right|_{\beta} ^{\beta}}}} \right)^2}}}\times\\
&{\frac{{{\sum\limits_{\alpha \in I_{r,m} {\left\{ {j}
\right\}}} {{{\rm{rdet}} _{j} {\left( {({\bf A} {\bf A}^{ *}
)_{j .} ({\bf w}^{(2)}  _{s  .} )}
\right)  _{\alpha} ^{\alpha} } }}}
}}{{{\sum\limits_{\alpha \in I_{r,m}} {{ {\left| {
{\bf A} {\bf A}^{ *} } \right| _{\alpha
}^{\alpha} }  }}} }}}=\nonumber\\=\nonumber&\sum_{z = 1}^m\sum_{s = 1}^n{\frac{\sum\limits_{\beta \in J_{r,n} {\left\{ {i}
\right\}}} {{\rm{cdet}} _{i} \left( {\left( {{\bf A}^{ *}
{\bf A}} \right)_{.i} \left({\bf w}^{(3)}_{.z} \right)}
\right)  _{\beta} ^{\beta} } }{{{\sum\limits_{\beta
\in J_{r,n}}  {{ \left| {{\bf A}^{ *}  {\bf
A}} \right|_{\beta}^{\beta} }}}} }}\times \\&\nonumber
{\frac{{    {\sum\limits_{\beta \in J_{r_1,\,m} {\left\{ {z}
\right\}}} {{\rm{cdet}} _{z} \left( {\left(\left({\bf V}^{ 2k+1} \right)^{*}{\bf V}^{ 2k+1} \right)_{. z} \left( \widetilde{ {\bf \psi}}_{.s}
\right)} \right)  _{\beta} ^{\beta} } }
}}{{\left({\sum\limits_{\beta \in J_{r_1,\,m}} {{\left| \left({\bf V}^{ 2k+1} \right)^{*}{\bf V}^{ 2k+1}
  \right|_{\beta} ^{\beta}}}} \right)^2}}}\times\\
&{\frac{\sum\limits_{\alpha \in I_{r,m} {\left\{ {j}
\right\}}} {{{\rm{rdet}} _{j} {\left( {({\bf A} {\bf A}^{ *}
)_{j .} ({\bf w}^{(2)}  _{s  .} )}
\right)  _{\alpha} ^{\alpha} } }
}}{{{\sum\limits_{\alpha \in I_{r,m}} {{ {\left| {
{\bf A} {\bf A}^{ *} } \right| _{\alpha
}^{\alpha} }  }}} }}}
\label{eq:det_wcmp_s2}
\end{align}
where ${\bf w}^{(3)}_{.z}$ is the $z$-th column of ${\bf W}_3:={\bf A}^*({\bf A}{\bf W})^{k+1}$ and ${\bf w}^{(2)}  _{s  .}$ is the $s$-th row of ${\bf W}_2:={\bf W}{\bf A}{\bf A}^*=(w^{(2)}_{sj})$.
Denote $$
 \widehat{ { \psi}}_{zs}:={\sum\limits_{\beta \in J_{r_1,\,m} {\left\{ {z}
\right\}}} {{\rm{cdet}} _{z} \left( {\left(\left({\bf V}^{ 2k+1} \right)^{*}{\bf V}^{ 2k+1} \right)_{. z} \left( \widetilde{ {\bf \psi}}_{.s}
\right)} \right)  _{\beta} ^{\beta} } }
$$ and construct the matrix $\widehat{\bf { \Psi}}=( \widehat{ { \psi}}_{zs})$. Then, introduce
\begin{align*}
\upsilon_{zj}=&\sum_{s = 1}^n\widehat{ { \psi}}_{zs}{\sum\limits_{\alpha \in I_{r,m} {\left\{ {j}
\right\}}} {{{\rm{rdet}} _{j} {\left( {({\bf A} {\bf A}^{ *}
)_{j .} ({\bf w}^{(2)}  _{s  .} )}
\right)  _{\alpha} ^{\alpha} } }
}}=\\=&{\sum\limits_{\alpha \in I_{r,m} {\left\{ {j}
\right\}}} {{{\rm{rdet}} _{j} {\left( {({\bf A} {\bf A}^{ *}
)_{j .} ({\bf { \psi}}^{(1)}  _{z  .} )}
\right)  _{\alpha} ^{\alpha} } }
}},
\end{align*}
where ${\bf { \psi}}^{(1)}  _{z  .}$ is the $z$-th row of ${\bf { \Psi}}^{(1)}=\widehat{\bf { \Psi}}{\bf W}_2=\widehat{\bf { \Psi}}{\bf W}{\bf A}{\bf A}^*$ and construct the matrix $\mathbf{\Upsilon}=(\upsilon_{zj})$.
Taking into account that
\begin{align*}
\sum_{z = 1}^n{\sum\limits_{\beta \in J_{r,n} {\left\{ {i}
\right\}}} {{\rm{cdet}} _{i} \left( {\left( {{\bf A}^{ *}
{\bf A}} \right)_{.i} \left({\bf w}^{(3)}_{.z} \right)}
\right)  _{\beta} ^{\beta} } }\upsilon_{zj}={\sum\limits_{\beta \in J_{r,n} {\left\{ {i}
\right\}}} {{\rm{cdet}} _{i} \left( {\left( {{\bf A}^{ *}
{\bf A}} \right)_{.i} \left(\widetilde{{\bf \upsilon}}_{.j} \right)}
\right)  _{\beta} ^{\beta} } },
\end{align*} where $\widetilde{{\bf \upsilon}}_{.j}$ is the $j$-th column of $\widetilde{{\bf \Upsilon}}={\bf W}_3{\bf \Upsilon}={\bf A}^{ *}({\bf A}{\bf W})^{k+1} {\bf \Upsilon}$,
finally from (\ref{eq:det_wcmp_s2}), it follows (\ref{eq:det_wcmp2}).
\end{proof}

Simpler expressions of determinantal representations of the WCMP inverse can be obtained in the cases having Hermicity.
\begin{thm}\label{th:detrep_cmph}Let ${\bf A}\in  {\mathbb{H}}^{m\times n}_r$ and ${\bf W}\in  {\mathbb{H}}^{n\times m}$ be a nonzero matrix. Suppose  $k=\max\{\Ind ({\bf W}{\bf A}),$ $\Ind ({\bf A}{\bf W})\}$.
 Then the  determinantal representations of its WCMP inverse $ {\bf A}^{c,{\dag},W}= \left(a_{ij}^{c,{\dag},W}\right)$ can be expressed as

(i)~ if ${\bf W}{\bf A}$ is Hermitian and $\rk ({\bf W}{\bf A})^k=r_1$, then
 \begin{align}
a_{ij}^{c,{\dag},W}=&\label{eq:det_wcmph1} {\frac{{\sum\limits_{\alpha \in I_{r,m} {\left\{ {j}
\right\}}} {{{\rm{rdet}} _{j} {\left( {({\bf A} {\bf A}^{ *}
)_{j .} (\widetilde{{\bf \omega}}  _{i  .} )}
\right)  _{\alpha} ^{\alpha} } }
}}
}{{\left({\sum\limits_{\beta \in J_{r,n}} {{ {\left| {
{\bf A}^{ *} {\bf A} } \right|_{\beta
}^{\beta} }  }}}\right)^2
{\sum\limits_{\alpha \in
I_{r_1,n}}  {{\left| {\left({\bf W} {{\rm {\bf A}} } \right)^{k+2}   } \right|_{\alpha} ^{\alpha}}}}
}}}\end{align}
where $\widetilde{{\bf \omega}}  _{i  .}$ is the $i$-th row  of $\widetilde{{\bf \Omega}}={\bf \Omega}{\bf W}{\bf A} {\bf A}^{ *}$. The matrix ${\bf \Omega}=(\omega_{is})$ is such that
\begin{align*}
\omega_{is}=&{\sum\limits_{\alpha
\in I_{r_1,n} {\left\{ {s} \right\}}} {{\rm{rdet}} _{s} \left(
{({\bf W}{\rm {\bf A}} )^{ k+2}_{s.} ( {\bf \phi}^{(1)}_{i.})}
\right)_{\alpha} ^{\alpha} } }.
\end{align*}
Here  ${\bf \phi}^{(1)}_{i.}$  is the $i$-th row of ${\bf \Phi}_1={\bf\Phi}{\bf A}\left({\bf W}{\bf A}\right)^k$  and the matrix ${\bf\Phi}=({\phi}_{it})$ is such that
\begin{align*}
{\phi}_{it}:={\sum\limits_{\beta \in J_{r,n} {\left\{ {i}
\right\}}} {{\rm{cdet}} _{i} \left( {\left( {{\bf A}^{ *}
{\bf A}} \right)_{.i} \left({\bf w}^{(1)}_{.t} \right)}
\right)  _{\beta} ^{\beta} } },
\end{align*}
where ${\bf w}^{(1)}_{.t}$ is the $t$-th column of ${\bf W}_1={\bf A}^*{\bf A}{\bf W}$.

(ii)~ if ${\bf A}{\bf W}$ is Hermitian and  $\rk ({\bf A}{\bf W})^k=r_1$, then
\begin{align}
a_{ij}^{c,{\dag},W}=&\label{eq:det_wcmph2}
 {\frac{\sum\limits_{\beta \in J_{r,n} {\left\{ {i}
\right\}}} {{\rm{cdet}} _{i} \left( {\left( {{\bf A}^{ *}
{\bf A}} \right)_{.i} \left(\widetilde{{\bf \upsilon}}_{.j} \right)}
\right)  _{\beta} ^{\beta} } }{{\left({\sum\limits_{\alpha \in I_{r,m}} {{ {\left| {
{\bf A} {\bf A}^{ *} } \right|_{\alpha
}^{\alpha} }  }}}\right)^2
{\sum\limits_{\beta \in J_{r_1,m}} {{\left|
{\left( { {\bf A}{\bf W}} \right)^{k+2}}  \right|_{\beta} ^{\beta}}} }
}}}
\end{align}
where $\widetilde{{\bf \upsilon}}_{.j}$ is the $j$-th column of $\widetilde{{\bf \Upsilon}}={\bf A}^{ *}{\bf A}{\bf W} {\bf \Upsilon}$. The matrix $\mathbf{\Upsilon}=(\upsilon_{tj})$ is determined by
\begin{align*}
\upsilon_{tj}=&{\sum\limits_{\beta
\in J_{r_1,\,m} {\left\{ {t} \right\}}} {{\rm{cdet}} _{t} \left(
{\left( {{\bf A}{\bf W}} \right)^{k+2}_{. t} \left({\bf { \psi}}^{(1)}_{.j}   \right)} \right) _{\beta}^{\beta} }  },
\end{align*}
 where ${\bf { \psi}}^{(1)}  _{.j}$ is the $j$-th column of ${\bf { \Psi}}^{(1)}=({\bf A}{\bf W})^{k}{\bf A}{\bf { \Psi}}$. Here ${\bf { \Psi}}=({ \psi}_{sj})$ is such that
 $$
  { \psi}_{sj}:={\sum\limits_{\alpha \in I_{r,m} {\left\{ {j}
\right\}}} {{{\rm{rdet}} _{j} {\left( {({\bf A} {\bf A}^{ *}
)_{j .} ({\bf w}^{(2)}  _{s  .} )}
\right)  _{\alpha} ^{\alpha} } }}},
$$
where  ${\bf w}^{(2)}  _{s  .}$ is the $s$-th row of ${\bf W}_2={\bf W}{\bf A}{\bf A}^*$.
\end{thm}
\begin{proof}(i)~
Taking into account (\ref{eq:wcmp}), applying  one of the cases of (\ref{eq:det_repr_proj_Q}) and (\ref{eq:det_repr_proj_P}) for the determinantal representations of   ${\bf Q}_A$ and ${\bf P}_A$, respectively,   and  (\ref{eq:dr_rep_wrdet}) for the determinantal representation of  ${\bf A}^{d,W}$ give

\begin{align}\nonumber
a_{ij}^{c,{\dag},W} =&\sum_{t = 1}^m\sum_{s = 1}^n{\frac{{{\sum\limits_{\beta \in J_{r,n} {\left\{ {i}
\right\}}} {{\rm{cdet}} _{i} \left( {\left( {{\bf A}^{ *}
{\bf A}} \right)_{.i} \left({\bf w}^{(1)}_{.t} \right)}
\right)  _{\beta} ^{\beta} } }}}{{{\sum\limits_{\beta
\in J_{r,n}}  {{ \left| {{\bf A}^{ *}  {\bf
A}} \right|_{\beta}^{\beta} }}}} }}\times \\&
 {\frac{\sum\limits_{\alpha
\in I_{r_1,n} {\left\{ {s} \right\}}} {{\rm{rdet}} _{s} \left(
{({\bf W}{\rm {\bf A}} )^{ k+2}_{s.} ( {\bf \bar{u}}_{t.} )}
\right)_{\alpha} ^{\alpha} } }{\sum\limits_{\alpha \in
I_{r_1,n}}  {{\left| {\left({\bf W} {{\rm {\bf A}} } \right)^{k+2}   } \right|_{\alpha} ^{\alpha}}}} }{\frac{{{\sum\limits_{\alpha \in I_{r,m} {\left\{ {j}
\right\}}} {{{\rm{rdet}} _{j} {\left( {({\bf A} {\bf A}^{ *}
)_{j .} ({\bf w}^{(2)}  _{s  .} )}
\right)  _{\alpha} ^{\alpha} } }}}
}}{{{\sum\limits_{\alpha \in I_{r,m}} {{ {\left| {
{\bf A} {\bf A}^{ *} } \right| _{\alpha
}^{\alpha} }  }}} }}},
\label{eq:det_wcmph_s}
\end{align}
where ${\bf w}^{(1)}_{.t}$ is the $t$-th column of ${\bf W}_1={\bf A}^*{\bf A}{\bf W}$,
${\bf \bar{u}}_{t.}$  is the $t$-th row of $\bar{{\bf U}}={\bf A}({\bf W}{\bf A})^k$, and ${\bf w}^{(2)}  _{s  .}$ is the $s$-th row of ${\bf W}_2={\bf W}{\bf A}{\bf A}^*$.

Denote
$$
{\phi}_{it}:={\sum\limits_{\beta \in J_{r,n} {\left\{ {i}
\right\}}} {{\rm{cdet}} _{i} \left( {\left( {{\bf A}^{ *}
{\bf A}} \right)_{.i} \left({\bf w}^{(1)}_{.t} \right)}
\right)  _{\beta} ^{\beta} } }
$$ and construct the matrix ${\bf\Phi}=({\phi}_{it})$. Then, determine
\begin{align*}
\omega_{is}=&\sum_{t = 1}^n{\phi}_{it} {\sum\limits_{\alpha
\in I_{r_1,n} {\left\{ {s} \right\}}} {{\rm{rdet}} _{s} \left(
{({\bf W}{\rm {\bf A}} )^{ k+2}_{s.} ( {\bf \bar{u}}_{t.} )}
\right)_{\alpha} ^{\alpha} } }=\\=&{\sum\limits_{\alpha
\in I_{r_1,n} {\left\{ {s} \right\}}} {{\rm{rdet}} _{s} \left(
{({\bf W}{\rm {\bf A}} )^{ k+2}_{s.} ( {\bf \phi}^{(1)}_{i.})}
\right)_{\alpha} ^{\alpha} } }
\end{align*}
where ${\bf \phi}^{(1)}_{i.}$ is the $i$-th row of ${\bf \Phi}_1={\bf\Phi}{\bf A}\left({\bf W}{\bf A}\right)^k$ and construct the matrix ${\bf\Omega}=(\omega_{is})$.
Taking into account that ${\sum\limits_{\alpha \in I_{r,m}} {{ {\left| {
{\bf A} {\bf A}^{ *} } \right| _{\alpha
}^{\alpha} }  }}}={\sum\limits_{\beta
\in J_{r,n}}  {{ \left| {{\bf A}^{ *}  {\bf
A}} \right|_{\beta}^{\beta} }}}$,
and
\begin{align*}
\sum_{s = 1}^n\omega_{is}{\sum\limits_{\alpha \in I_{r,m} {\left\{ {j}
\right\}}} {{{\rm{rdet}} _{j} {\left( {({\bf A} {\bf A}^{ *}
)_{j .} ({\bf w}^{(2)}  _{s  .} )}
\right)  _{\alpha} ^{\alpha} } }
}}={\sum\limits_{\alpha \in I_{r,m} {\left\{ {j}
\right\}}} {{{\rm{rdet}} _{j} {\left( {({\bf A} {\bf A}^{ *}
)_{j .} (\widetilde{{\bf \omega}}  _{i  .} )}
\right)  _{\alpha} ^{\alpha} } }
}},
\end{align*} where $\widetilde{{\bf \omega}}  _{i  .}$ is the $i$-th row of $\widetilde{{\bf \Omega}}={\bf \Omega}{\bf W}_2={\bf \Omega}{\bf W}{\bf A}{\bf A}^{ *}$,
finally from (\ref{eq:det_wcmph_s}), it follows (\ref{eq:det_wcmph1}).

(ii) By applying (\ref{eq:dr_rep_wrdet})  for the determinantal representation of  ${\bf A}^{d,W}$ and the same determinantal representations of   ${\bf Q}_A$ and ${\bf P}_A$ as in the  point (i), we get
\begin{align}\nonumber
a_{ij}^{c,{\dag},W} =&\sum_{t = 1}^m\sum_{s = 1}^n{\frac{{{\sum\limits_{\beta \in J_{r,n} {\left\{ {i}
\right\}}} {{\rm{cdet}} _{i} \left( {\left( {{\bf A}^{ *}
{\bf A}} \right)_{.i} \left({\bf w}^{(1)}_{.t} \right)}
\right)  _{\beta} ^{\beta} } }}}{{{\sum\limits_{\beta
\in J_{r,n}}  {{ \left| {{\bf A}^{ *}  {\bf
A}} \right|_{\beta}^{\beta} }}}} }}\times \\&
{\frac{\sum\limits_{\beta
\in J_{r_1,\,m} {\left\{ {t} \right\}}} {{\rm{cdet}} _{t} \left(
{\left( {{\bf A}{\bf W}} \right)^{k+2}_{. t} \left( {{\bf
\bar{v}}_{.s} }  \right)} \right) _{\beta}
^{\beta} }  }{\sum\limits_{\beta \in J_{r_1,m}} {{\left|
{\left( { {\bf A}{\bf W}} \right)^{k+2}}  \right|_{\beta} ^{\beta}}} }}{\frac{{{\sum\limits_{\alpha \in I_{r,m} {\left\{ {j}
\right\}}} {{{\rm{rdet}} _{j} {\left( {({\bf A} {\bf A}^{ *}
)_{j .} ({\bf w}^{(2)}  _{s  .} )}
\right)  _{\alpha} ^{\alpha} } }}}
}}{{{\sum\limits_{\alpha \in I_{r,m}} {{ {\left| {
{\bf A} {\bf A}^{ *} } \right| _{\alpha
}^{\alpha} }  }}} }}},
\label{eq:det_wcmph_s2}
\end{align}
where ${\bf w}^{(1)}_{.t}$ is the $t$-th column of ${\bf W}_1:={\bf A}^*{\bf A}{\bf W}$,  ${\bf \bar{v}}_{.s} $ is the $s$-th column of  ${\bf \bar{V}}=({\bf A}{\bf W})^{k}{\bf A} $,  and ${\bf w}^{(2)}  _{s  .}$ is the $s$-th row of ${\bf W}_2:={\bf W}{\bf A}{\bf A}^*$.
Denote $$
  { \psi}_{sj}:={\sum\limits_{\alpha \in I_{r,m} {\left\{ {j}
\right\}}} {{{\rm{rdet}} _{j} {\left( {({\bf A} {\bf A}^{ *}
)_{j .} ({\bf w}^{(2)}  _{s  .} )}
\right)  _{\alpha} ^{\alpha} } }}}
$$ and construct the matrix ${\bf { \Psi}}=(  { \psi}_{sj})$. Then, introduce
\begin{align*}
\upsilon_{tj}=&\sum_{s = 1}^n{\sum\limits_{\beta
\in J_{r_1,\,m} {\left\{ {t} \right\}}} {{\rm{cdet}} _{t} \left(
{\left( {{\bf A}{\bf W}} \right)^{k+2}_{. t} \left( {{\bf
\bar{v}}_{.s} }  \right)} \right) _{\beta}
^{\beta} }  } { \psi}_{sj}=\\=&{\sum\limits_{\beta
\in J_{r_1,\,m} {\left\{ {t} \right\}}} {{\rm{cdet}} _{t} \left(
{\left( {{\bf A}{\bf W}} \right)^{k+2}_{. t} \left({\bf { \psi}}^{(1)}_{.j}   \right)} \right) _{\beta}^{\beta} }  },
\end{align*}
where ${\bf { \psi}}^{(1)}  _{.j}$ is the $j$-th column of ${\bf { \Psi}}^{(1)}=({\bf A}{\bf W})^{k}{\bf A}{\bf { \Psi}}$ and construct the matrix $\mathbf{\Upsilon}=(\upsilon_{tj})$.
Taking into account that
\begin{align*}
\sum_{t = 1}^n{\sum\limits_{\beta \in J_{r,n} {\left\{ {i}
\right\}}} {{\rm{cdet}} _{i} \left( {\left( {{\bf A}^{ *}
{\bf A}} \right)_{.i} \left({\bf w}^{(1)}_{.t} \right)}
\right)  _{\beta} ^{\beta} } }\upsilon_{tj}={\sum\limits_{\beta \in J_{r,n} {\left\{ {i}
\right\}}} {{\rm{cdet}} _{i} \left( {\left( {{\bf A}^{ *}
{\bf A}} \right)_{.i} \left(\widetilde{{\bf \upsilon}}_{.j} \right)}
\right)  _{\beta} ^{\beta} } },
\end{align*} where $\widetilde{{\bf \upsilon}}_{.j}$ is the $j$-th column of $\widetilde{{\bf \Upsilon}}={\bf A}^{ *}{\bf A}{\bf W} {\bf \Upsilon}$,
finally from (\ref{eq:det_wcmph_s2}), it follows (\ref{eq:det_wcmph2}).
\end{proof}

\begin{cor}Let ${\bf A}\in  {\mathbb{C}}^{m\times n}_r$ and ${\bf W}\in  {\mathbb{C}}^{n\times m}$ be a nonzero matrix. Suppose  $k=\max\{\Ind ({\bf W}{\bf A}),$ $\Ind ({\bf A}{\bf W})\}$.
 Then the  determinantal representations of its WCMP inverse $ {\bf A}^{c,{\dag},W}= \left(a_{ij}^{c,{\dag},W}\right)$ can be expressed as
 
 (ii)~ if   $\rk ({\bf W}{\bf A})^k=r_1$, then
 \begin{align*}
a_{ij}^{c,{\dag},W}=& {\frac{{\sum\limits_{\alpha \in I_{r,m} {\left\{ {j}
\right\}}} {{ {\left| {({\bf A} {\bf A}^{ *}
)_{j .} (\widetilde{{\bf \omega}}  _{i  .} )}
\right|  _{\alpha} ^{\alpha} } }
}}
}{{\left({\sum\limits_{\beta \in J_{r,n}} {{ {\left| {
{\bf A}^{ *} {\bf A} } \right|_{\beta
}^{\beta} }  }}}\right)^2
{\sum\limits_{\alpha \in
I_{r_1,n}}  {{\left| {\left({\bf W} {{\rm {\bf A}} } \right)^{k+2}   } \right|_{\alpha} ^{\alpha}}}}
}}}\end{align*}
where $\widetilde{{\bf \omega}}  _{i  .}$ is the $i$-th row  of $\widetilde{{\bf \Omega}}={\bf \Omega}{\bf W}{\bf A} {\bf A}^{ *}$. The matrix ${\bf \Omega}=(\omega_{is})$ is such that
\begin{align*}
\omega_{is}=&{\sum\limits_{\alpha
\in I_{r_1,n} {\left\{ {s} \right\}}} { \left|
{({\bf W}{\rm {\bf A}} )^{ k+2}_{s.} ( {\bf \phi}^{(1)}_{i.})}
\right|_{\alpha} ^{\alpha} } }.
\end{align*}
Here  ${\bf \phi}^{(1)}_{i.}$  is the $i$-th row of ${\bf \Phi}_1={\bf\Phi}{\bf A}\left({\bf W}{\bf A}\right)^k$  and the matrix ${\bf\Phi}=({\phi}_{it})$ is such that
\begin{align*}
{\phi}_{it}:={\sum\limits_{\beta \in J_{r,n} {\left\{ {i}
\right\}}} { \left| {\left( {{\bf A}^{ *}
{\bf A}} \right)_{.i} \left({\bf w}^{(1)}_{.t} \right)}
\right| _{\beta} ^{\beta} } },
\end{align*}
where ${\bf w}^{(1)}_{.t}$ is the $t$-th column of ${\bf W}_1={\bf A}^*{\bf A}{\bf W}$.

(ii)~ if   $\rk ({\bf A}{\bf W})^k=r_1$, then
\begin{align*}
a_{ij}^{c,{\dag},W}=&
 {\frac{\sum\limits_{\beta \in J_{r,n} {\left\{ {i}
\right\}}} { \left| {\left( {{\bf A}^{ *}
{\bf A}} \right)_{.i} \left(\widetilde{{\bf \upsilon}}_{.j} \right)}
\right|  _{\beta} ^{\beta} } }{{\left({\sum\limits_{\alpha \in I_{r,m}} {{ {\left| {
{\bf A} {\bf A}^{ *} } \right|_{\alpha
}^{\alpha} }  }}}\right)^2
{\sum\limits_{\beta \in J_{r_1,m}} {{\left|
{\left( { {\bf A}{\bf W}} \right)^{k+2}}  \right|_{\beta} ^{\beta}}} }
}}}
\end{align*}
where $\widetilde{{\bf \upsilon}}_{.j}$ is the $j$-th column of $\widetilde{{\bf \Upsilon}}={\bf A}^{ *}{\bf A}{\bf W} {\bf \Upsilon}$. The matrix $\mathbf{\Upsilon}=(\upsilon_{tj})$ is determined by
\begin{align*}
\upsilon_{tj}=&{\sum\limits_{\beta
\in J_{r_1,\,m} {\left\{ {t} \right\}}} {\left|
{\left( {{\bf A}{\bf W}} \right)^{k+2}_{. t} \left({\bf { \psi}}^{(1)}_{.j}   \right)} \right| _{\beta}^{\beta} }  },
\end{align*}
 where ${\bf { \psi}}^{(1)}  _{.j}$ is the $j$-th column of ${\bf { \Psi}}^{(1)}=({\bf A}{\bf W})^{k}{\bf A}{\bf { \Psi}}$. Here ${\bf { \Psi}}=({ \psi}_{sj})$ is such that
 $$
  { \psi}_{sj}:={\sum\limits_{\alpha \in I_{r,m} {\left\{ {j}
\right\}}} {{ {\left| {({\bf A} {\bf A}^{ *}
)_{j .} ({\bf w}^{(2)}  _{s  .} )}
\right|  _{\alpha} ^{\alpha} } }}},
$$
where  ${\bf w}^{(2)}  _{s  .}$ is the $s$-th row of ${\bf W}_2={\bf W}{\bf A}{\bf A}^*$.
\end{cor}

Theorems \ref{th:detrep_cmp} and \ref{th:detrep_cmph} give  determinantal representations of the WCMP inverse over the quaternion skew field.
For better understanding, we present the algorithm of its finding, for example, in Theorem \ref{th:detrep_cmp} the  case (i). Other  algorithms   can be construct similarly.

\begin{alg}
\begin{enumerate}

    \item Compute the matrix $\check{ {\bf U}}=
{\bf U}^{k}({\bf U}^{ 2k+1})^{*}.$
    \item Find     $\phi_{lq}$ by (\ref{eq:phi1}) for all $l=1,\ldots,n$ and $q=1,\ldots,n$ and construct the matrix $
{\bf \Phi}=(\phi_{lq})$.
        \item Compute the matrix $
\widetilde{ {\bf \Phi}}:={\bf A}{\bf \Phi}{\bf U}^{2k}({\bf U}^{ 2k+1})^{*}$.
           \item Find     $\widehat{\phi}_{tz}$ by (\ref{eq:phih}) for all $t=1,\ldots,m$ and $z=1,\ldots,n$ and construct the matrix $\widehat{\Phi}=(\widehat{\phi}_{tz})$.
         \item Compute the matrix ${\bf \Phi}_1={\bf A}^*{\bf A}{\bf W}\widehat{\Phi}$.
 \item By   (\ref{eq:om2}), find  $\omega_{iz}$   for all $i=1,\ldots,m$ and $z=1,\ldots,n$ and construct the matrix ${\bf \Omega}=(\omega_{iz})$.

         \item Finally, find $a_{ij}^{c,{\dag},W}$  by  (\ref{eq:det_wcmp1}) for all $i=1,\ldots,m$ and $j=1,\ldots,n$ .

  \end{enumerate}
\end{alg}

\section{An example}

In this section, we give an example to illustrate our results. Given  the matrices
\begin{align}\label{eq:giv_ex}{\bf A}=\begin{bmatrix}
  0 & \mathbf{i} & 0 \\
  \mathbf{k} & 1 & \mathbf{i} \\
 1 & 0 & 0\\
  1 & -\mathbf{k} & -\mathbf{j}
\end{bmatrix},\,\, {\bf W}=\begin{bmatrix}
 \mathbf{k} & 0 & \mathbf{i} & 0 \\
 -\mathbf{j} & \mathbf{k} & 0 & 1 \\
 0 & 1 & 0 & -\mathbf{k}
\end{bmatrix}.\end{align}
Since
\begin{align*}{\bf V}=&{\bf A}{\bf W}=\begin{bmatrix}
 -\mathbf{k} & -\mathbf{j} & 0 & \mathbf{i} \\
 -1-\mathbf{j} & \mathbf{i}+\mathbf{k} & \mathbf{j} & 1+\mathbf{j} \\
 \mathbf{k} & 0 & \mathbf{i} & 0\\
 -\mathbf{i}+\mathbf{k} & 1-\mathbf{j} & \mathbf{i} & \mathbf{i}-\mathbf{k}
\end{bmatrix},~{\bf U}={\bf W}{\bf A}=\begin{bmatrix}
  \mathbf{i} & \mathbf{j} & 0 \\
  0 & \mathbf{k} & 0 \\
 0 & 0 & 0
\end{bmatrix},\\
{\bf A}^*{\bf A}=&\begin{bmatrix}
3& -2\mathbf{k} & -2\mathbf{j}  \\
2\mathbf{k} &3& 2\mathbf{i} \\
2\mathbf{j}&-2\mathbf{i} &2
\end{bmatrix},~
{\bf A}{\bf A}^*=\begin{bmatrix}
1& \mathbf{i} &0& -\mathbf{j}  \\
- \mathbf{i} &3& \mathbf{k} & 3\mathbf{k} \\
0& -\mathbf{k} & 1&1\\
\mathbf{j}&-3\mathbf{k} & 1 & 3
\end{bmatrix}
\end{align*}
and $\rk{\bf A}=3$, $\rk{\bf W}=3$, $\rk{\bf V}=3$, $\rk{\bf V}^{3}=\rk {\bf V}^{2}=2$, $\rk{\bf U}^{2}=\rk {\bf U}=2$, then $ \Ind {\bf V}=2$, $ \Ind {\bf U}=1$, and $k= {\max}\{\Ind({\bf A}{\bf W}), \Ind({\bf W}{\bf A})\}=2$.

We shall find the weighted DMP inverse due to Algorithm \ref{al1}.
\begin{enumerate}
 \item
Compute the matrix $\check{ {\bf U}}=
{\bf U}^{2}({\bf U}^{ 5})^{*}.$
Since \begin{align*}{\bf U}^{2}=&\begin{bmatrix}
  -1 & \mathbf{i}+\mathbf{k} & 0 \\
  0 & -1 &  \\
 0 & 0 & 0
\end{bmatrix},\,\,\,{\bf U}^{5}=\begin{bmatrix}
  \mathbf{i} & 2+3\mathbf{j} & 0 \\
  0 & \mathbf{k} &  \\
 0 & 0 & 0
\end{bmatrix},\end{align*}then \begin{align*}
\check{{\rm {\bf U}}}=&({\bf U}^{ 5})^{*}{\bf U}^{2}=\begin{bmatrix}
  i & 1+j & 0 \\
  -2+3j & -i+6k &  \\
 0 & 0 & 0
\end{bmatrix}\end{align*}
and $\rk {\bf U}^{2}=2.$

    \item By (\ref{eq:phidmp}) find     $\phi_{iq}$  for all $i,q=1,2,3$. So,  $
{\bf \Phi}=\begin{bmatrix}
  \mathbf{i} & -2-\mathbf{j} & 0 \\
  0 & \mathbf{k} &  0\\
 0 & 0 & 0
\end{bmatrix}.$

        \item Further,  the matrix $
\widehat{ {\bf \Phi}}:={\bf W}{\bf A}{\bf \Phi}{\bf U}^{4}({\bf U}^{ 5})^{*}=\begin{bmatrix}
  6\mathbf{i}-\mathbf{k} & 1+\mathbf{j} & 0 \\
  -2+3\mathbf{j} &\mathbf{k}  &  0\\
 0 & 0 & 0
\end{bmatrix}.$

 \item By   (\ref{eq:om1}) find  $\omega_{is}$   for all $i,s=1,2,3$. So, we have that ${\bf \Omega}={\bf \Phi}$.
\item Compute the matrix $
\widetilde{ {\bf \Omega}}:={\bf \Omega}{\bf U}^{3}{\bf A}^{*}=\begin{bmatrix}
  0 & -\mathbf{k} &  1&1\\
  -\mathbf{i} &1&0 & \mathbf{k}  \\
 0 & 0 & 0&0
\end{bmatrix}.$
         \item Finally, find $a_{ij}^{d,{\dag},W}$  by  (\ref{eq:det_wdmp}) for all $i=1,\ldots,4$ and $j=1,2,3$.
So,
 \begin{align*}a_{11}^{d,{\dag},W}=
{\frac{{\sum\limits_{\alpha \in I_{3,4} {\left\{ {1}
\right\}}} {{{\rm{rdet}} _{1} {\left( {({\bf A} {\bf A}^{ *}
)_{1 .} (\widetilde{{\bf \omega}}  _{1  .} )}
\right)  _{\alpha} ^{\alpha} } }
}}
}{{{\sum\limits_{\beta \in I_{3,4}} {{ {\left| {
{\bf A} {\bf A}^{ *} } \right|_{\alpha
}^{\alpha} }  }}}
{{\left(
{\sum\limits_{\alpha \in I_{2,3}} {{\left|  {\bf U}^{5} \left({\bf U}^{5} \right)^{*}
 \right|_{\alpha} ^{\alpha}}}}\right)^2 }},
}}}=\\=\frac{1}{2}\left({\rm{rdet}} _{1}\begin{bmatrix}
         0& -\mathbf{k}&1\\
         -\mathbf{i}&3& \mathbf{k} \\
          0 & -\mathbf{k}&1
      \end{bmatrix}+{\rm{rdet}} _{1}\begin{bmatrix}
         0&1&1\\
           0&1&1\\
          \mathbf{j}&-3\mathbf{k}&3
      \end{bmatrix}\right.\\\left.+{\rm{rdet}} _{1}\begin{bmatrix}
         0&-\mathbf{k}&1\\
           -\mathbf{i}&3&3\mathbf{k}\\
          \mathbf{j}&-3\mathbf{k}&3
      \end{bmatrix}\right)=0.\end{align*}

\end{enumerate}
Continuing similarly, we obtain
 \begin{equation}\label{eq:Af}{\bf A}^{d,{\dag},W}=\begin{bmatrix}
      0& 0&1&0 \\
        -\mathbf{i} &0&0&0 \\
      0&0&0&0\\
      \end{bmatrix}.\end{equation}
It is easy verify that ${\bf X}={\bf A}^{d,{\dag},W}$ from (\ref{eq:Af}) with the given matrices (\ref{eq:giv_ex}) is the solution to Eqs. (\ref{eq:prop_wdmp}).

\section{Conclusions}\label{sec:con}
\noindent  Notions of  the weighted core-EP right and left inverses, the weighted DMP and MPD inverses, and the weighted CMP inverse have been extended to quaternion matrices  in this paper. Due to noncommutativity of quaternions, these generalized inverses in quaternion matrices  have some features in comparison to complex matrices. We have obtained their determinantal representations within the framework of the theory of column-row determinants previously introduced by the author. As the special cases,  their determinantal representations in complex matrices have been obtained as well.

\end{document}